\documentclass[11pt]{amsart}
\usepackage{amssymb, latexsym, amsmath, amsfonts}
\usepackage[pdftex]{graphicx}
\usepackage{float}
\newtheorem{thm}{Theorem}[section]
\newtheorem{cor}[thm]{Corollary}
\newtheorem{lem}[thm]{Lemma}
\newtheorem{prop}[thm]{Proposition}
\theoremstyle{definition}

\theoremstyle{remark}

\numberwithin{equation}{section}

\newcommand{\mbb}{\mathbb}
\newcommand{\ra}{\rightarrow}

\newcommand{\ep}{\epsilon}
\newcommand{\no}{\noindent}
\newcommand{\al}{\alpha}

\newcommand{\De}{\Delta}

\newcommand{\ti}{\tilde}
\newcommand{\la}{\lambda}
\newcommand{\ka}{\kappa}
\newcommand{\norm}[1]{\left\Vert#1\right\Vert}

\newcommand{\Aut}{{\rm Aut}(\mathbb C^k)}

\begin{document}
\title{Some aspects of shift--like automorphisms of $\mbb C^k$}
\thanks{$^1$ Supported by CSIR-UGC(India) fellowship.}
 \thanks{$^2$ Supported by the DST SwarnaJayanti
Fellowship 2009--2010 and a UGC--CAS Grant.}

\author{Sayani Bera$^{1}$  and Kaushal Verma$^{2}$} 
\address{Sayani Bera: Department of Mathematics, Indian Institute of Science,
Bangalore 560 012, India}
\email{sayani@math.iisc.ernet.in}

\address{Kaushal Verma: Department of Mathematics, Indian Institute of Science,
Bangalore 560 012, India}
\email{kverma@math.iisc.ernet.in}

\pagestyle{plain}

\begin{abstract}
The goal of this article is two fold. First, using transcendental shift--like automorphisms of $\mbb C^k, ~k \ge 3$ we construct two examples of non--degenerate entire mappings with 
prescribed ranges. The first example exhibits an entire mapping of $\mbb C^k, ~k \ge 3$ whose range avoids a given polydisc but contains the complement of a slightly larger 
concentric polydisc. This generalizes a result of Dixon--Esterle in $\mbb C^2.$ The second example shows the existence of a Fatou--Bieberbach domain in $\mbb C^k,~k \ge 3$ that is 
constrained to lie in a prescribed region. This is motivated by similar results of Buzzard and Rosay--Rudin. 
In the second part we compute the order and type of
entire mappings that parametrize one dimensional unstable manifolds for shift-like
polynomial automorphisms and show how they can be used to prove a Yoccoz type
inequality for this class of automorphisms.
\end{abstract}

\maketitle

\section{Introduction}

\noindent Let ${\rm Aut}(\mbb C^k)$ be the group of holomorphic automorphisms of
$\mbb C^k$. For $k \ge 2$, ${\rm Aut}(\mbb C^k)$ is infinite dimensional as it
contains the
automorphisms defined by
\[
(z_1, z_2, \ldots, z_k) \mapsto (z_1, \ldots, z_{k -1}, h(z_1, z_2, \ldots,
z_{k-1}) + z_k)
\]
where $h$ is an entire function on $\mbb C^{k-1}$. Let ${\rm Aut}_1(\mbb C^k)
\subset {\rm Aut}(\mbb C^k)$ be the subgroup consisting of automorphisms $f$
whose Jacobian $J_f \equiv 1$.
Thus ${\rm Aut_1}(\mbb C^k)$ consists exactly of those automorphisms that preserve
volume. The topology on ${\rm Aut}(\mbb C^k)$ is the topology of uniform
convergence on compact subsets
of $\mbb C^k$ applied to both a map and its inverse. This gives ${\rm Aut}(\mbb
C^k)$ the structure of a Frechet space and in fact this topology can be induced
by a complete metric
which may be quickly recalled as follows. For $r>0$ and $f \in {\rm Aut}(\mbb C^k)$, let
$\norm{f}_r$ denote the supremum norm of $f$ on the closed
ball of radius $r$ centered at origin. For each $n \ge 1$ and $f, g \in
\Aut$, let
\[
d_n(f, g) = \max \Big \{ \norm{f(z) - g(z)}_n,\norm{f^{-1}(z) - g^{-1}(z)}_n\Big\}
\]
and define
\[
d(f, g) = \sum_{n \ge 1} 2^{-n} d_n(f, g)/(1 + d_n(f, g)).
\]

\medskip

For $k \ge 2$ and $1 \le \nu \le k - 1$, a shift-like mapping of type $\nu$ is
an element of $\Aut$ defined by
\[
f(z_1, z_2, \ldots, z_k) = (z_2, \ldots, z_k, h(z_{k - \nu + 1}) - a z_1)
\]
for some entire function $h$ and $a \in \mbb C^{\ast}$. Such a map is determined
by a choice of an entire function $h$ and a non-zero constant $a$. We shall also consider a finite
composition of such maps $f
= f_m \circ \ldots \circ f_1$ where each $f_s$ is shift-like of type $\nu$ and
is determined by an entire function $h_s$ and $a_s \in \mbb C^{\ast}$. 
The automorphisms $f$ obtained by choosing $h$ to be a
polynomial were introduced and studied by Bedford--Pambuccian in \cite{BP} and
it was noted that they
share many properties with the much studied class of H\'{e}non maps in 
${\rm Aut}(\mbb C^2)$ such as the existence of a filtration, the construction of Green
functions and the
corresponding stable and unstable currents. These shift-like maps can therefore
be thought of as generalizations of H\'{e}non maps and in fact each $f_s$ is
regular in the sense
of Sibony (\cite{Sb}). However, the composition of more than one shift-like maps of type $\nu$ is only
$q$-regular, for some $q \ge 1$, as defined by Guedj--Sibony in \cite{GS}.

\medskip

Let $f_j \in {\rm Aut}_1(\mbb C^k)$ be a sequence that converges uniformly on
compact subsets of $\mbb C^k$ to $f$. It is evident that $f$ is injective and
$J_f \equiv 1$. In fact,
Dixon--Esterle observed that (see \cite{DE}) the range of $f$ is precisely those
points $p \in \mbb C^k$ for which the set of preimages $\{f^{-1}_j(p)\}$ is
bounded. Using this
observation they constructed several non-degenerate holomorphic endomorphisms of
$\mbb C^2$ with prescribed ranges -- here, a non-degenerate map is one whose
Jacobian does
not vanish identically. To cite an example, Theorem 8.13 in \cite{DE} constructs
a non-degenerate holomorphic endomorphism $G$ of $\mbb C^2$ for every given $\ep
> 0$ such that
the range $G(\mbb C^2)$ avoids the closure of the unit polydisc $D^2(0, 1)
\subset \mbb C^2$ but contains the complement of the slightly large polydisc
$D^2(0, 1 + \ep) = \{z \in
\mbb C^2 : \vert z_1 \vert, \vert z_2 \vert < 1 + \ep\}$. Moreover, for each $w
\in G(\mbb C^2)$, there is a bound on the cardinality of the fibre $G^{-1}(w)$
in terms of $\ep$ alone.
This should be compared with Theorem 8.5 of Rosay--Rudin (\cite{RR1}) that
constructs for a given strictly convex compact set $K \subset \mbb C^k$ ($k \ge
2$) and a countable dense
subset $E$ of $\mbb C^n \setminus K$, a volume preserving biholomorphism from
$\mbb C^k$ to $\mbb C^k$ whose range avoids $K$ but contains $E$. At the heart
of the construction in
\cite{DE} is a sequence in ${\rm Aut}_1(\mbb C^2)$ of the form \[
\theta_m(z_1, z_2) = (u_1, u_2)
\]
where
\begin{align}
u_1 &= z_1 + h_m(z_2)\\
u_2 &= z_2 + h_m(u_1) \notag
\end{align}
and the $h_m$'s are judiciously chosen entire functions with prescribed
behaviour in apriori fixed vertical strips in the complex plane. The sequence
$G_m = \theta_1 \circ \theta_2
\circ \ldots \circ \theta_m$ is shown to converge uniformly on compact subsets
of $\mbb C^2$ to, say $\ti G$ whose range can be understood by keeping track of
the conditions imposed on
the $h_m$'s. Doing this is the main part of the construction and it sheds a
great deal of light on the behaviour of $\ti G$. Upto an affine change of
coordinates, the desired map $G =
\psi \circ \tilde G$ where $\psi(z) = (z_1^n, z_2^n)$ for a suitable $n \ge 1$.
The genesis of this note lies in our attempt to understand this construction in
\cite{DE} with a view to
constructing similar examples in $\mbb C^k$, $k \ge 3$. For $r > 0$, let
$D^k(0, r) \subset \mbb C^k$ be the polydisc around the origin with radius $r$
in each of the coordinate
directions. It must be mentioned that this technique of constructing entire
mappings with prescribed ranges by starting with a suitable sequence in ${\rm
Aut}_1(\mbb C^k)$ has been
used rather successfully in the recent past and examples include \cite{Gl},
\cite{GlSt}, \cite{St} and \cite{W} to cite only a few.

\begin{thm}\label{thm1.1}
For every $k \ge 3$ and every $\ep > 0$, there is a non-degenerate holomorphic
mapping $G : \mbb C^k \ra \mbb C^k$ such that the range $G(\mbb C^k)$ avoids the
closure of the unit
polydisc $D^k(0, 1)$ but contains the complement of the larger polydisc
$D^k(0, 1 + \ep)$. Moreover, for each $w \in G(\mbb C^k)$, the cardinality of
the fibre $G^{-1}(w)$ is at
most $n^k$ where $n$ depends only on $\ep > 0$.
\end{thm}

The first step in the proof of this theorem is to identify a sequence in ${\rm
Aut}_1(\mbb C^k)$ that has the right dynamical properties to create the desired
mapping $G$. To do this,
note that $\theta_m$ as in (1.1) can be factored as
\[
\theta_m = \tau_m \circ \tau_m
\]
where
\[
\tau_m(z_1, z_2) = (z_2, z_1 + h_m(z_2))
\]
is a shift-like map of type $1$. Thus we are led to consider the sequence
\[
\Theta_m = \eta_m \circ \eta_m \circ \ldots \circ \eta_m
\]
where
\[
\eta_m(z_1, z_2, \ldots, z_k) = (z_2, \ldots, z_k, z_1 + h_m(z_k))
\]
is a shift-like map of type 1 and $h_m$ is a sequence of entire functions with
prescribed behaviour in apriori fixed vertical strips. The conditions imposed on
the $h_m$'s guarantee the
uniform convergence of $G_m = \Theta_1 \circ \Theta_2 \circ \ldots \circ
\Theta_m$ on compact subsets of $\mbb C^k$ and since the method of \cite {DE}
does not directly generalize to the case $k \ge 3,$ Section $2$ extends and
streamlines the proof in \cite{DE} in
$\mbb C^2$ to obtain the desired holomorphic endomorphism of $\mbb C^k$ for all
$k \ge 3$.
\medskip
\par

The next result is motivated by the constructions of Fatou-Bieberbach domains
that are constrained to lie in apriori prescribed small regions in $\mbb C^k$
due to Buzzard \cite{Bu}
and Rosay--Rudin \cite{RR2}. In particular, Theorem 4.3 in \cite{Bu} shows that
for a given lattice $\Lambda \in \mbb C^k$, $k \ge 2$ there exists an $\ep > 0$
small enough
and a Fatou-Bieberbach domain that lies in the complement of the union of the
balls $B(q, \ep)$ as $q$ runs over all points in $\Lambda$. By working with
transcendental shift-like
maps of type $1$ again, the basic idea of \cite{DE} as discussed above can be
adapted to show the existence of a volume preserving endomorphism of $\mbb C^k$,
$k \ge 2$, whose
range avoids large open sets and a specified collection of disjoint polydiscs, but
at the same time contains an apriori given collection of disjoint polydiscs as well.
To make this precise,
fix $a > 0$ and for an integer $K \ge 1$, define:
\medskip

\begin{figure}[H]\label{fig1}
\includegraphics[height=4in,width=7in]{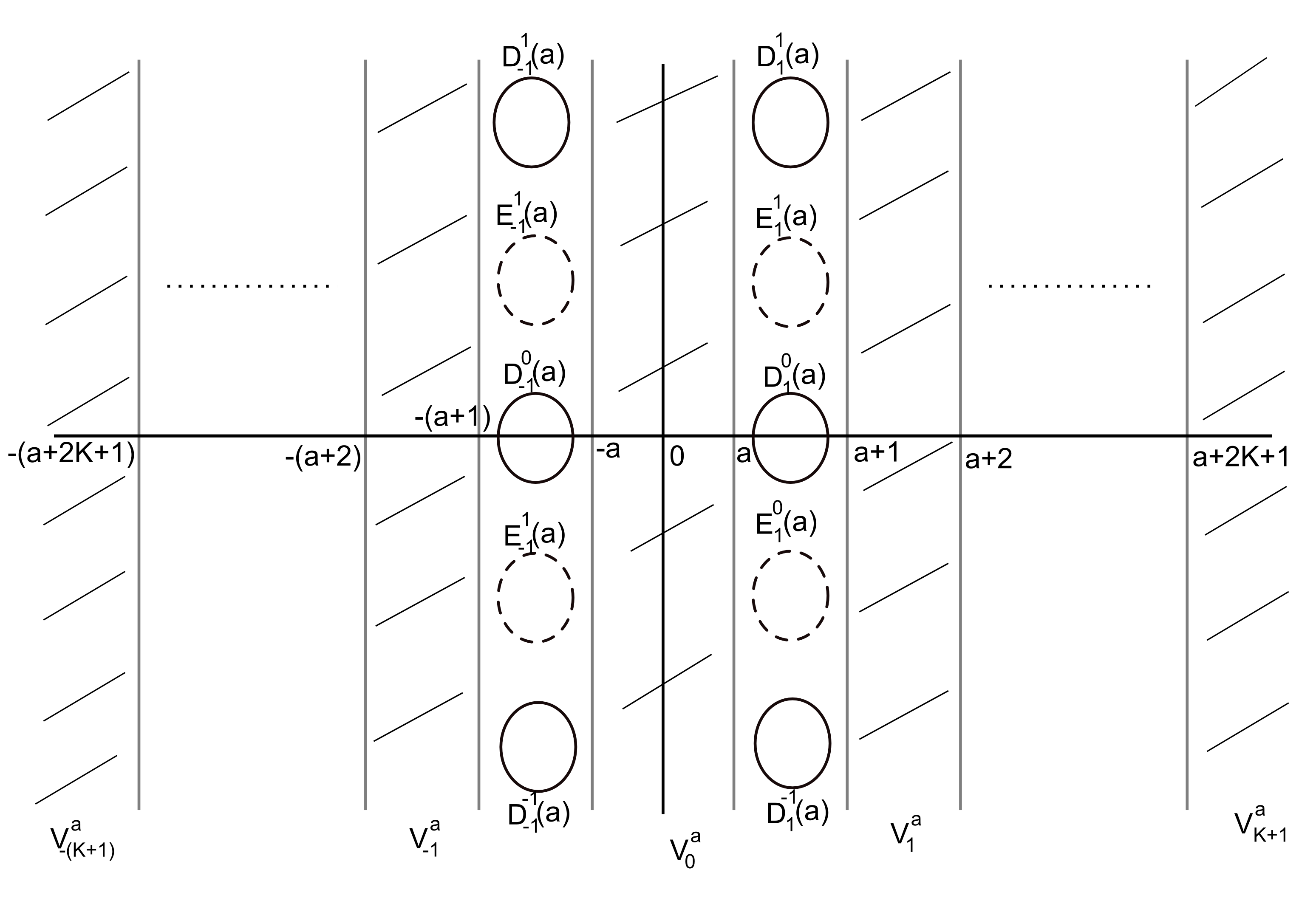}
\caption{}
\end{figure}
\begin{itemize}

\item $V^a_l = \big\{ z \in \mbb C : a + (2l-1) \le \Re z \le a + 2l \big\},\; 1
\le l \le K$,
\medskip
\item $V^a_{K + 1} = \big\{ z \in \mbb C : \Re z \ge a + 2K + 1 \big\}$,
\medskip
\item $V^a_0 = \big\{ z \in \mbb C : \vert \Re z \vert \le a \big\}$,
\medskip
\item $V^a_{-l} = \big\{ z \in \mbb C : -(a + 2l) \le \Re z \le -(a + 2l - 1)
\big\}, \; 1 \le l \le K$,
\medskip
\item $V^a_{-(K+1)} = \big\{ z \in \mbb C : \Re z \le -(a + 2K + 1)  \big\}$,
\medskip
\item $D^a_l = D^0_l(a) \cup D^1_{\pm l}(a) \cup D^{\pm 2}_l(a) \cup \ldots$, where for each $n \in \mbb Z$ and
$1 \le \vert l \vert \le K + 1$, $D^n_l(a)$ is a disc of radius $1/8$ centered at the
point $(a + (4l -3)/2 ,
2n)\in \mbb C$, and
\medskip
\item $E^a_l = E^0_l(a) \cup E^{\pm 1}_l(a) \cup E^{\pm 2}_l(a) \cup \ldots$, where for each $n \in \mbb Z$ and
$1 \le \vert l \vert \le K + 1$, $E^n_l(a)$ is a disc of radius $1/8$ centered at the
point $(a + (4l -3)/2,
2n+1) \in \mbb C$.

\end{itemize}

\medskip
As seen in Figure 1, the regions $V^a_{\pm l}, V^a_0, V^a_{\pm(K + 1)}$ are
pairwise disjoint vertical strips in the plane and the discs $D^n_l(a),
E^n_l(a)$, which are pairwise disjoint
as well, are contained in the corridors between them.

\begin{thm}\label{thm1.2}
For every $a >0$ and $K \ge 1$, there exists an injective, volume preserving holomorphic map
$F : \mbb C^k \ra \mbb C^k$ such that
\begin{enumerate}
\item[(i)] $V^a_{l_1} \times V^a_{l_2} \times \ldots \times V^a_{l_k} \notin
F(\mbb C^k)$, where $1 \le l_i \le K + 1$ for every $1 \le i \le k$,
\item[(ii)] $V^a_{-l_1} \times V^a_{-l_2} \times \ldots \times V^a_{-l_k} \notin
F(\mbb C^k)$, where $1 \le l_i \le K + 1$ for every $1 \le i \le k$,
\item[(iii)] $D^a_{l_1} \times D^a_{l_2} \times \ldots \times D^a_{l_k} \notin
F(\mbb C^k)$, where $\vert l_i \vert \le K + 1$ for every $1 \le i \le k$, and
\item[(iv)]  $E^a_{l_1} \times E^a_{l_2} \times \ldots \times E^a_{l_k} \in
F(\mbb C^k)$, where $\vert l_i \vert \le K + 1$ for every $1 \le i \le k$.
\end{enumerate}
\end{thm}

\noindent We will henceforth focus on polynomial shift-like maps of type $1$.
Let $F = F_m \circ F_{m-1} \circ \ldots \circ F_1$, $m \ge 2$ where each
\[
F_i(z_1, z_2, \ldots, z_k) = (z_2, z_3, \ldots, z_k, \al_i z_1 + p_i(z_k))
\]
for some polynomial $p_i$ of degree $d_i \ge 2$ and $\al_i \not= 0$. The degree
of $F$ is $d = d_1d_2 \ldots d_m$. Let
\[
K^{\pm} = \big\{ z \in \mbb C^k : \text{the sequence} \; F^{\pm n}(z) \;
\text{is bounded} \big\}
\]
and $K = K^+ \cap \;K^-$. Suppose $a \in \mbb C^k$ is a saddle point of $F$ such
that the eigenvalues $\la_1, \la_2, \ldots, \la_k$ of the derivative $DF(a)$
satisfy $\vert \la_k \vert
> 1$ and $\vert \la_i \vert < 1$ for all $1 \le i \le k -1$. Such an $a$ will be
called a $(k-1, 1)$ saddle. The existence of a filtration for $F$ was
established in \cite{BP} and here
we are interested in using it to study the unstable manifolds associated to a
$(k-1, 1)$ saddle. Now given a $(k-1, 1)$ saddle $a$, we will drop the subscript
$k$ on the
unique eigenvalue which is greater than $1$ in modulus and simply write it as
$\la$. By Sternberg's theorem, there is an
entire mapping $H : \mbb C \ra \mbb C^k$ such that the unstable manifold $W^u(a)
= H(\mbb C)$ with $H(0) = a$ and $F \circ H(t) = H(\la t)$ for $t \in \mbb C$.
Note that $W^u(a) \subset
K^-$. Let $H(t) = (h_1(t), h_2(t), \ldots, h_k(t))$. Recall that the order
$\rho$ of $g \in \mathcal O(\mbb C)$ is
\[
\rho = {\rm ord} g = \limsup_{r \ra \infty} \frac{\log \log \sup_{ \vert z \vert
= r} \vert g(z) \vert}{\log r}
\]
and if $0 < \rho < \infty$, we define the type $\sigma$ of $g$ by
\[
\sigma = \limsup_{r \ra \infty}(1/r^{\rho}) \log {\rm sup}_{ \vert z \vert = r} \vert
g(z) \vert.
\]
We say that $g$ is of mean type if $0 < \sigma < \infty$. It was shown by
Fornaess--Sibony \cite{FS} that for degree two H\'{e}non maps in $\mbb C^2$,
each of the coordinate functions
$h_1, h_2$ of the map $H$ that parametrizes the unstable manifold of a saddle
point, with eigenvalues $0 < \vert \mu \vert < 1 < \vert \la \vert$, are entire
functions of order at most
$\log 2 / \log \vert \la \vert$. This was strengthened by Jin (\cite{J}) who
showed that the order is in fact equal to $\log d / \log \vert \la \vert$ where
$d$ is the degree of a given
generalized H\'{e}non map. A related result was also obtained by Cantat
(\cite{C}).

\medskip

For $z = (z_1, z_2, \ldots, z_k) \in \mbb C^k$, let $\Vert z \Vert = \max_{1 \le
i \le k} \vert z_i \vert$ and for large $R < \infty$ define
\begin{align*}
V_i &= \big\{ z = (z_1, z_2, \ldots, z_k) \in \mbb C^k : \vert z_i \vert > R,
\vert z_i \vert = \Vert z \Vert \big\} \; \text{for} \;1 \le i \le k,\\
V   &= \big\{ z \in \mbb C^k : \Vert z \Vert \le R  \big\},\\
V^+ &= V_k \; \text{and} \; V^- = V_2 \cup V_3 \cup \ldots \cup V_{k-1}.
\end{align*}
Let $G_j = F_j \circ F_{j-1} \circ \ldots \circ F_1$ for all $1 \le j \le m$.
The coordinate projection on the $j$-th factor will be denoted by $\pi_j$, i.e.,
$\pi_j(z) = \pi_j(z_1,
z_2, \ldots, z_k) = z_j$ for $1 \le j \le k$ and for a positive integer $n$, let
$[n] = n \mod m$. For a given $\ep > 0$ and every $1 \le i \le k$, define $m-1$
constants by
\[
D^l_i = d_{[m-k+i]} d_{[m-k+i-1]} \ldots d_{[m-k+i-(l-1)]}
\]
where $1 \le l \le m-1$. Finally, for $n \ge 0$ and $1 \le i \le k$, let $z_i^n
= \pi_i \circ F^n(z)$.

\begin{thm}\label{thm1.3}
For a given $\ep > 0$ there exist $2(k-1)$ constants defined by
\[
\log C_j(\pm \ep) = (1 + D^1_j + D^2_j + \ldots + D^{m-1}_j) \log (1 \pm \ep)
\]
where $2 \le j \le k$ and $M >0$ such that if $z = (z_1, z_2, \ldots, z_k) \in
K^-$, then for every $n \ge 0$
\[
\vert z_j^n \vert \le C_j(\ep)^{d^n} \max \Big\{ \vert z_j^0 \vert^{d^n},
M^{d^n} \Big\}
\]
and
\[
\max \Big\{ \vert z_j^n \vert, C_j(-\ep)^{d^n} M^{d^n} \Big\} \ge C_j(-\ep)^{d^n}
\vert z_j^0 \vert^{d^n}.
\]
\end{thm}

\noindent As a consequence, for the map $F$ as above that has a $(k-1, 1)$ saddle point 
$a \in \mbb C^k$ and a corresponding parametrization $H$ of
$W^u(a)$, we have the following:

\begin{thm}\label{thm1.4}
The coordinates $h_i$ ($1 \le i \le k$) of the parametrization $H$ are
transcendental entire functions of mean type of order $\rho = \log d /\log \vert
\la \vert$.
\end{thm}

\noindent Let $\ti K = H^{-1}(K)$. Following Jin (\cite{J}), we say that $\ti K$
is {\it bridged} if it does not contain the origin as an isolated point. In this case,
as the order of the
$h_i$'s is finite, the Denjoy--Carleman--Ahlfors theorem implies that $\mbb C
\setminus \ti K$ consists of a finite number of components, say $q'$. These
components may be given a
circular ordering, say $U_1, U_2, \ldots, U_q'$, and multiplication by $\la$
preserves the ordering on the collection of components. This means that there is
an integer $p'$ such that
$\la U_j \subset U_{j + p'}$ for every $1 \le j \le q'$. Let $N = \gcd(p', q')$
and $p = p'/N, q = q'/N$. The reasoning given in \cite{J} yields the following
Yoccoz-type inequality
for shift-like maps of type $1$ and we will briefly indicate the main steps for
the sake of completeness. An enlightening discussion of the relevance of such
Yoccoz-type inequalities
may be found in \cite{H} while \cite{BS} discusses such an inequality for
H\'{e}non maps.

\begin{thm}\label{thm1.5}
If $\ti K$ is bridged, there is a choice of logarithm $\tau = \log \la$ such
that
\[
\frac{\Re \tau}{\vert \tau - 2 \pi i p/q \vert^2} \ge \frac{Nq}{2 \log d}.
\]
\end{thm}

\par

\medskip
Note that Proposition A in \cite{J} gives a criterion for $\tilde{K}$ to be bridged in the case of H\'{e}non maps in $\mbb C^2.$ 
It is not clear to us whether an analogous 
criterion is valid for shift-line maps in $\mbb C^k,~k \ge 3.$
\section{Proof of Theorem ~\ref{thm1.1}}
\noindent
As mentioned before, the main idea here is to construct a sequence of
automorphisms $\{ f_j \} \in {\rm Aut}_1 (\mbb{C}^k)$, $k \geq 3$ and use the
following observation in \cite{DE} to obtain injective, volume preserving
holomorphic maps with non-dense range.
\begin{lem}\label{1}
Let $\{f_n\}_{n \geqslant 1}$ be sequence of elements of ${\rm
Aut}_1(\mathbb{C}^k)$ such that $f_n \to F$ uniformly on compact subsets of
$\mbb{C}^k$. Then
\begin{enumerate}
  \item[(i)] $F$ is injective,
  \item[(ii)] $F(\mathbb{C}^k)$ is a Runge domain,
  \item[(iii)] $ f_n^{-1}(z)$ converges uniformly to $F^{-1}(z)$ on compact
subsets of $F(\mathbb{C}^k)$, and
  \item[(iv)] $|{f_n^{-1}(z)}| \to \infty $ uniformly on $\mathbb{C}^k\diagdown
F(\mathbb{C}^k)$.
\end{enumerate}
\end{lem}
\medskip
\par
 We first define an element in ${\rm Aut}_1(\mbb C^k)$ which is essentially a
prescribed translation in apriori specified domains in $\mbb{C}^k$. To loosely
define these domains, take $k$ copies of $\mbb C$ and divide each copy of $\mbb
C$ into two pieces by the vertical line say $\Re z=0$ . Then $\mbb C^k$ is the
union of $2^k$ disjoint domains and their boundaries, where each domain is
created by taking the product of $k$ copies of half planes, either the left or
right half plane in each factor. Note that each domain is unbounded and the boundary
of each domain is piecewise smooth, the smooth faces being open pieces of affine
real hypersurfaces. By the Arakelian--Gauthier theorem \cite{AG} it is possible
to construct an automorphism in ${\rm Aut}_1(\mbb C^k)$ that behaves like an
apriori specified translation in this domain. Call this Step 0. Now we
iterate the procedure. In Step 1 we dissect the $k$ copies of $\mbb
C$ by say $\Re z=1$. This gives rise to another decomposition of $\mbb C^k$ into
$2^k$ disjoint domains, and just as before, we can apply \cite{AG} again to construct
another element of ${\rm Aut}_1(\mbb C^k)$, which behaves as a translation in
each of these domains. Note that these newly created $2^k$ domains intersect the
previous collection of domains. This gives rise to a finer decomposition of $\mbb
C^k$ by taking all possible intersections of domains in Step 1 and Step 0. For every
$n \geq 1$, Step $n$
is then unambigously defined. The limit of the compositions of these elements in
${\rm Aut}_1(\mbb C^k)$ which are created at each step is of interest to us. The
main difficulty in controlling the sequence of composition is to understand its
behaviour on the finer and finer domains that are created at each step of this
inductive procedure.
\medskip
 \par
 The following lemma establishes the existence of maps in ${\rm Aut}_1(\mbb
C^k)$ as described above when $k=3.$ The case when $k \geq 4 $ is stated
separately in the Appendix for the purpose of clarity.
 \begin{thm}\label{thm1a}
Let $\rho(z)=\min\{|\pi_1(z)|,|\pi_2(z)|,|\pi_3(z)|\}$. Let $\delta>0$ and $0<
\eta < \delta /2.$ For every pair of real triples $\lambda_1,\lambda_2,\lambda_3
$ and $a_1,a_2,a_3,$ there exists a map $\theta \in Aut_1(\mbb{C}^k)$ such that
\begin{enumerate}
 \item [(i)]$|\theta(z)-z+(\lambda_1, \lambda_2, \lambda_3)|<\eta
e^{-\rho(z)^{1/4}}$ whenever
 \[
  \Re \pi_{1}(z)\geqslant a_{1}+\delta,~\Re \pi_{2}(z)\geqslant
a_{2}+\delta,~\Re \pi_{3}(z)\geqslant a_{3}+\delta.
 \]
 \item [(ii)]$|\theta(z)-z+(\lambda_1,0,\lambda_3)|<\eta e^{-\rho(z)^{1/4}}$
whenever
     \[
 \Re \pi_{1}(z)\leqslant a_{1}-\delta,~\Re \pi_{2}(z)\geqslant a_{2}+\delta,~\Re
\pi_{3}(z)\geqslant a_{3}+\delta.
 \]
 \item [(iii)]$|\theta(z)-z+(0,\lambda_2,\lambda_3)|<\eta e^{-\rho(z)^{1/4}}$
whenever
 \[
 \Re \pi_{1}(z)\geqslant a_{1}+\delta,~\Re \pi_{2}(z)\geqslant a_{2}+\delta,~\Re
\pi_{3}(z)\leqslant a_{3}-\delta.\]
 \item [(iv)]$|\theta(z)-z+(0,0,\lambda_3)|<\eta e^{-\rho(z)^{1/4}}$ whenever
 \[
 \Re \pi_{1}(z)\leqslant a_{1}-\delta,~\Re \pi_{2}(z)\geqslant a_{2}+\delta,~\Re
\pi_{3}(z)\leqslant a_{3}-\delta.\]
 \item [(v)]$|\theta(z)-z+(\lambda_1, \lambda_2, 0)|<\eta e^{-\rho(z)^{1/4}}$
whenever
 \[
  \Re \pi_{1}(z)\geqslant a_{1}+\delta,~\Re \pi_{2}(z)\leqslant
a_{2}-\delta,~\Re \pi_{3}(z)\geqslant a_{3}-\lambda_3+\delta.
 \]
 \item [(vi)]$|\theta(z)-z+(\lambda_1,0,0)|<\eta e^{-\rho(z)^{1/4}}$ whenever
     \[
 \Re \pi_{1}(z)\leqslant a_{1}-\delta,~\Re \pi_{2}(z)\leqslant a_{2}-\delta,~\Re
\pi_{3}(z)\geqslant a_{3}-\lambda_3+\delta.
 \]
 \item [(vii)]$|\theta(z)-z+(0,\lambda_2,0)|<\eta e^{-\rho(z)^{1/4}}$ whenever
 \[
 \Re \pi_{1}(z)\geqslant a_{1}+\delta,~\Re \pi_{2}(z)\leqslant a_{2}-\delta,~\Re
\pi_{3}(z)\leqslant a_{3}-\lambda_3-\delta.\]
 \item [(viii)]$|\theta(z)-z+(0,0,0)|<\eta e^{-\rho(z)^{1/4}}$ whenever
 \[
 \Re \pi_{1}(z)\leqslant a_{1}-\delta,~\Re \pi_{2}(z)\leqslant a_{2}-\delta,~\Re
\pi_{3}(z)\leqslant a_{3}-\lambda_3-\delta.\]
 \end{enumerate}
 \end{thm}
 \begin{figure}[H]\label{fig2}
\includegraphics[height=4in,width=7in]{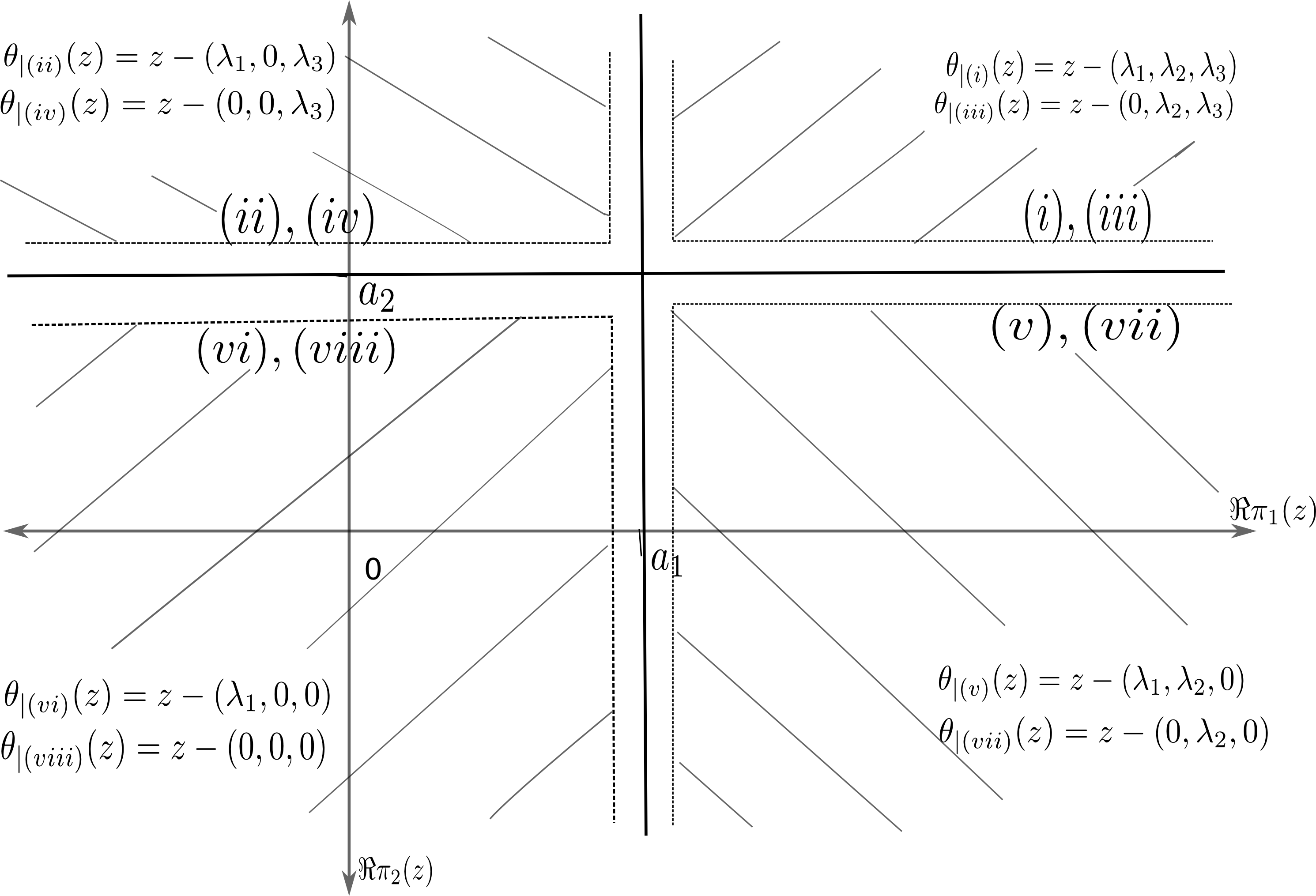}
\caption{}
\end{figure}
\noindent
Figure \ref{fig2} depicts the regions described by the various conditions in the theorem above. These regions are tube domains
since their defining conditions involve only the real parts of the coordinate functions. It is therefore sufficient to describe
the base of the regions. Note that the base is a domain in $\mbb R^3$ where the coordinates are $\Re \pi_1(z),\Re \pi_2(z)$ and
$\Re \pi_3(z).$ The $\Re \pi_1(z)$ and $\Re \pi_2(z)$ axes had been indicated while the $\Re \pi_3(z)$ axis is to be understood as
pointing outwards. With this convention we have:
\begin{itemize}
\item For region (i) and (ii), $\Re \pi_3(z) \ge a_3+ \delta,$
\item For region (iii) and (iv), $\Re \pi_3(z) \le a_3- \delta,$
\item For region (v) and (vi), $\Re \pi_3(z) \ge a_3- \lambda_3+\delta,$
\item For region (vii) and (viii), $\Re \pi_3(z) \le a_3-\lambda_3- \delta.$
\end{itemize}
\par
\begin{proof}
With $\delta>0$, $0 < \eta < \delta/2$ and $n \gg 1$ there exist entire
functions $f_n^j$ for every $1 \leq j  \leq 3$ satisfying
\begin{align*}
 \mbox{(i)} \hspace{10 mm} &|f_n^1(z)-\lambda_1| \leqslant\frac{1}{n}
e^{-|z|^{1/4}} <\eta e^{-|z|^{1/4}} ~~\mbox{if}~~\Re z\geqslant
a_3-\lambda_3+\delta/2 ,~\mbox{and} \\
&|f_n^1(z)-0| \leqslant \frac{1}{n} e^{-|z|^{1/4}}<\eta
e^{-|z|^{1/4}}~~\mbox{if}~~ \Re z\leqslant a_3-\lambda_3-\delta/2. &\\
\mbox{(ii)} \hspace{10 mm} &|f_n^j(z)-\lambda_n| \leqslant \frac{1}{n}
e^{-|z|^{1/4}} <\eta e^{-|z|^{1/4}}~~\mbox{if}~~\Re z\geqslant a_{j-1}+\delta/2
, ~\mbox{and}\\
&|f_n^j(z)-0| \leqslant \frac{1}{n} e^{-|z|^{1/4}} <\eta
e^{-|z|^{1/4}}~~\mbox{if}~~\Re z\leqslant a_{j-1}-\delta/2, ~~\mbox{where}~~
j=2,3.
\end{align*}

The existence of such entire functions with the above
properties follows from the Arakelian- Gauthier theorem \cite{AG} (see \cite{DE} as well) which says that
if $V \subset \mbb C \cup \{\infty \}  \simeq \mbb P^1 $ is a closed subset of $\mbb C$, such that its
complement is connected and locally connected at $\{\infty\}$ then any continuous function
$h:V \rightarrow \mathbb{C}$ which is  holomorphic in the interior of $V$ can be
uniformly approximated on $V$ by entire functions with any apriori specified
accuracy. More precisely, if $\omega$ is a positive continuous function on
$\mathbb{R}^+$ such that
\[
    \int\limits_1^{\infty} \frac{|{\log\omega(t)}|}{t^{3/2}} dt < +\infty
\]
then there exists an entire function $g$ over $\mathbb{C}$ such that
$|{h(z)-g(z)}| \leqslant \omega(|z|)$ for every $z \in V.$
\par
\medskip
\noindent
We apply this theorem to construct, say $f_n^1$. Take $V$ to be the union of the
closed disjoint half planes specified by the condition in (i). The function
\begin{align*}
h(z)= \left\{
           \begin{array}{ll}
             \lambda_1,&\hbox{if $\Re z \geq a_3-\lambda_3+\delta/2$;} \\
             0,&\hbox{if $\Re z \leq a_3-\lambda_3-\delta/2$.}
           \end{array}
         \right.
\end{align*}
can therefore be approximated by entire functions $f_n^1$ that satisfy the
desired conditions. A similar argument shows the existence of $f_n^2$ and
$f_n^3.$
Define $\theta_n \in {\rm Aut}_1(\mbb{C}^3)$ as
$\theta_n(z_1,z_2,z_3)=(v_1,v_2,v_3)$ where
\begin{align*}
v_1 &= z_1 + f_{n}^1(v_3),\\
v_2 &= z_2 + f_{n}^2(z_1), \\
v_3 &= z_3 + f_{n}^3(z_2).
\end{align*}
A straight forward calculation shows that $\theta_n= (S_3 \circ S_3 \circ
S_1)^{-1}$ where each $S_j$ is a type 1 shift of the form
\[
 S_j(z_1,z_2,z_3)=(z_2,z_3,z_1+f_n^j(z_3)) ~\mbox{for all}~~ 1\leq j \leq 3.
\]
\noindent
Now let $z \in \mbb{C}^3$ be such that $\Re\pi_{1}(z)\geqslant a_{1}+\delta ,\;
\Re\pi_{2}(z)\geqslant a_{2}+\delta ,\; \Re\pi_3(z)\geqslant
a_3-\lambda_3+\delta.$
Then $$|f_n^j(z_{j-1})-\lambda_j|=|\pi_j(\theta_n(z))-\pi_j(z)+\lambda_j|< \eta
e^{-\rho(z)^{1/4}}$$
for $j=2,3$ . So
$$\Re v_3\geqslant \Re z_3-\lambda_3-\eta\geqslant a_3-\lambda_3+\delta/2.$$
Thus
$$|\pi_1(\theta_n(z))-\pi_1(z)+\lambda_1|=|f_n^1(v_3)-\lambda_1| <
e^{-|v_3|^{1/4}}/n.$$ 
Since $|v_3|>|z_3|-|\lambda_3|-1$ and $t^{1/4}-(t-\lambda_3
-1)^{1/4}\to 0$ as $t \to \infty$ there exists a constant $C>0$ such that
\[e^{-|v_3|^{1/4}}<C e^{-|z_3|^{1/4}}.\]
If we choose $n$ large enough so that
$C/n<\eta$ and $1/n<\eta$, then $\theta_n$ satisfies condition (i) of the theorem.
Similarly it can be shown that for large $n$, $\theta_n$ satisfies all the other
conditions.
 \end{proof}

For $k\geqslant3$ and $n\geqslant 1$, let
\begin{align*}
\Lambda_j^1(n)=\Big\{z\in \mbb{C}^k: \Re\pi_j(z)>n\Big\}
~\mbox{and }\Lambda_j^0(n)=\Big\{z\in \mbb{C}^k: \Re\pi_j(z)<n\Big\}
\end{align*}
for $1 \leq j \leq k $ and define $2^k$ subsets of $\mbb{C}^k$ as
\begin{align*}
&Y^n(i_1,i_2,\hdots,0,0)=\bigcap_{j=1}^{k-2}\Lambda_j^{i_j}(n)\cap\Lambda_{k-1}
^0(n)\cap \Lambda_k^0(n-1),\\
&Y^n(i_1,i_2,\hdots,0,1)=\bigcap_{j=1}^{k-2}\Lambda_j^{i_j}(n)\cap\Lambda_{k-1}
^0(n)\cap \Lambda_k^1(n-1),\\
&Y^n(i_1,i_2,\hdots,1,0)=\bigcap_{j=1}^{k-2}\Lambda_j^{i_j}(n)\cap\Lambda_{k-1}
^1(n)\cap \Lambda_k^0(n),\mbox{~and}\\
&Y^n(i_1,i_2,\hdots,1,1)=\bigcap_{j=1}^{k-2}\Lambda_j^{i_j}(n)\cap\Lambda_{k-1}
^1(n)\cap \Lambda_k^1(n),
\end{align*}

where $i_j \in \{0,1\}$ and $1 \le j \le k-2.$

\medskip

Let \[\widetilde{Y}^n=\Big\{z\in \mbb{C}^k: z\in
Y^n(i_1,i_2,\hdots,i_k),\;i_j\in\{0,1\}\; ~\mbox{and}~ 1 \le j \le k\Big\}.\] Define
$P^n:\widetilde{Y}^n \to \mbb C^k$ by

 \[
P^n(z)= z-(i_k,i_1,i_2,\hdots,i_{k-1})
\]
for $z \in {Y^n(i_1,i_2,\hdots,i_k)}.$
\par
\medskip
\noindent
It is evident that the map $P^n$ is injective on $\widetilde{Y}^n$
and $P^n(z)=z$ on $Y^n(0,0,\hdots,0)$ for every $n.$
Also $$P^n (z)= z-(1,1,\hdots ,1)$$
for $z \in {Y^n(1,1, \hdots,1)},$ and so we have
\begin{align}\label{eq0}
P^n(Y^n(1,1,...,1))=Y^{n-1}(1,1,\hdots,1) 
\end{align}
for every $ n \ge 2.$
Let $l_j\in \mbb N$ for $1 \le j \le k-1$ and $l_k \in \{0\} \cup \mbb{N}.$ Define
$\Gamma_j^{l_j},\Gamma_k^{l_k} \subset \mbb{C}^k$ as
\begin{enumerate}
\item[(i)]$\Gamma_j^{l_j}=\Big\{z\in \mbb{C}^k: l_j-1<\Re\pi_j(z)<l_j ~\mbox{if}~
l_j>1~\mbox{else}~\Re\pi_j(z)<1 \Big\}$ for all $1 \le j \le k-1.$
\item[(ii)]$\Gamma_k^{l_k}=\Big\{z\in \mbb{C}^k: l_k-1<\Re\pi_k(z)<l_k ~\mbox{if}~
l_k>0~\mbox{else}~\Re\pi_k(z)<0 \Big\}.$
\end{enumerate}
Let $\Delta(l_1,\hdots,l_k)=\bigcap\limits_{j=1}^k \Gamma_j^{l_j}$ and
\[
\widetilde{\Delta}= \bigcup \Delta(l_1,l_2, \hdots ,l_k)
\]
where the union is taken over all possible $l_j \ge 1$ for $1 \le j \le k-1$ and $l_k \ge 0.$ 
\par
\medskip
\noindent
Then $\widetilde{\Delta}\subset\widetilde{Y}^n$ and hence the map $P^n$ is
defined on $\widetilde{\Delta}$ for every $n.$
Also note that $P^n(\widetilde{\Delta})\subset \widetilde{\Delta}.$
\par
\begin{figure}[H] \label{fig3}
\includegraphics[height=4in,width=6in]{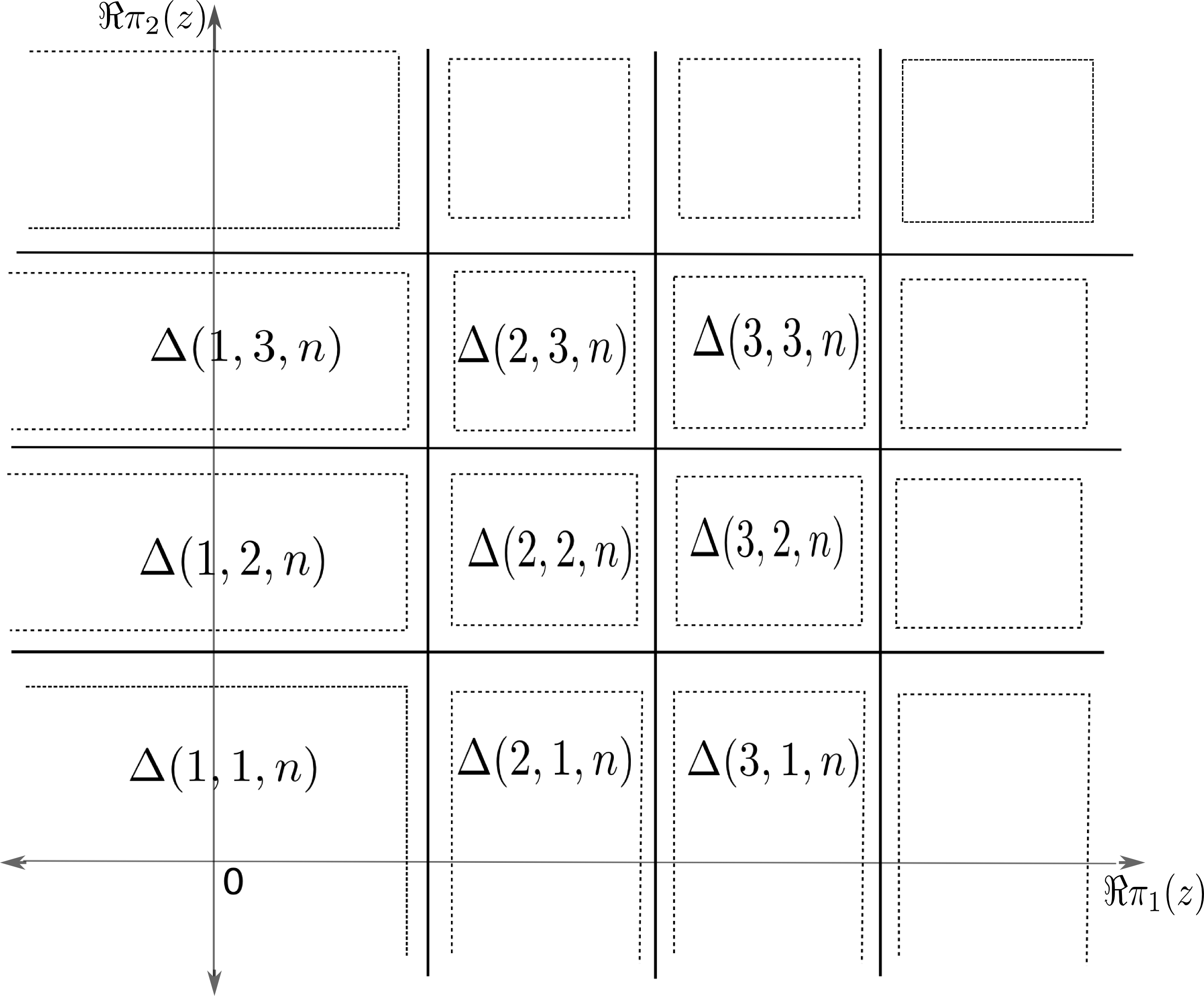}
\caption{}
\end{figure}
\noindent
Figure 3 here represents the sets $\De(l_1,l_2, \hdots, l_k)$ when $k=3.$
\par
\medskip
\noindent
Now suppose that $z\in \widetilde{\Delta}$. Then $z \in \Delta(l_1,\hdots,l_k)$
for some $l_j \ge 1$ for $1 \le j \le k-1$ and $l_k \ge 0$. If $l_z=\max\{l_1,l_2,\hdots,l_k\}$ and $m \ge l_z$, then
$z \in Y^m(0,\hdots,0).$ So
\[
P^m(z)=z~
\]
whenever $m>l_z.$
\par
\medskip
\noindent
Therefore the map defined by $S(z)=P^1\circ P^2 \circ \hdots\circ P^{l_z}(z)$ is well
defined for every $z\in \widetilde{\Delta}$ and hence on $\widetilde{\Delta}.$
Let $\Omega=S(\widetilde{\Delta}).$

\begin{lem}\label{lem1}
$\Omega \cap \Big\{ z \in \mbb{C}^k: \Re\pi_j(z)>0 ~{\rm for}~
1 \le j \le k \Big \}=\emptyset.$
\end{lem}
\begin{proof}
Let $z_0 \in \Omega \cap \{ z \in \mbb{C}^k: \Re\pi_j(z)>0 ~\mbox{for}
~1 \le j \le k\}$. Then there exists $w_0 \in \widetilde{\Delta}$ such that
$S(w_0)=z_0$. As $w_0 \in \widetilde{\Delta}$ it follows that $w_0 \in
\Delta(l_1,l_2,\hdots,l_k)$ for some $k$-tuple $(l_1,l_2,\hdots ,l_k)$ and if $l_{w_0}=\max\{l_1,l_2,\hdots,l_k\}$ then
$$z_0=S(w_0)=P^1\circ P^2 \circ \hdots\circ P^{l_{w_0}}(w_0),$$ i.e.,
$$w_0=(P^{l_{w_0}})^{-1}\circ\hdots\circ(P^2)^{-1}\circ(P^1)^{-1}(z_0).$$ Since
\[P^n(Y^n(1,1,...,1))=Y^{n-1}(1,1,\hdots,1)
\]
for every $n \ge 2$ and when $n=1,$ it follows that
\[
P^1(Y^1(1,1,\hdots,1))=\{z \in \mbb{C}^k: \Re\pi_j(z)>0 ~\mbox{for} ~1 \le j \le k \}
\]
which implies $\Re\pi_j(w_0)>l_{w_0} ~\mbox{for}
~1 \le j \le k$. This is a contradiction!
\end{proof}
\medskip
Define $T$ on $\Omega$ as $T=S^{-1}_{|\Omega}$ and let $A$ and $B$ be sets
defined as
\begin{equation*}
 A=\Big\{z\in \mbb{C}^k\setminus\Omega: \min_{1 \le l \le k}\{\Re\pi_l(z)\}\leqslant
0\Big\},~B=\mbb{C}^k\setminus\widetilde{\Delta}.
\end{equation*}
\begin{lem}
Both $A$ and $B$ have empty interior.
\end{lem}
\begin{proof}
That $B$ has empty interior is evident by the definition of $\widetilde{\Delta}.$ Now let
\begin{align*}
C_l =\Big\{ z \in \mbb{C}^k: \min_{1 \le j \le k}\{\Re \pi_j(z))\}\leqslant 0 ~\mbox{and}~
\Re \pi_l(z) \in \mbb{N}\cup \{0\}\Big\}
\end{align*}
for every $1 \le l \le k$. Then observe that each $C_l$ has empty
interior.
Suppose there exists $z \in A$ such that $z \notin C_l$ for any
$l$, then
\begin{align*}
\Re \pi_j(z) \notin \mbb{N} \cup \{0\} ~\mbox{for} ~1 \le j \le k
~\mbox{and}~\min_{1 \le j \le k}\{\Re\pi_j(z)\}< 0.
\end{align*}
{\it Claim: }There exists $n_z \in \mbb{N}$ such that $$\max~\Big\{\Re
\pi_j((P^{n_z})^{-1}\circ\hdots\circ(P^2)^{-1}\circ(P^1)^{-1}(z))\Big\} \le
n_z.$$
Let $$w_n=(P^{n})^{-1}\circ\hdots\circ(P^2)^{-1}\circ(P^1)^{-1}(z).$$
Then note that there exists $j_0 \in \{1,2,\hdots,k\}$ such that $\Re
\pi_{j_0}(w_1)<1.$ This implies that
$$ \Re \pi_{j_0}(w_n)<n ~\mbox{and}~\Re \pi_{j_0+1}(w_1)=\Re \pi_{j_0+1}(w_n)$$
for $n > 1.$ When $j_0=k,$ then the above conditions are to be understood as
$$ \Re \pi_{k}(w_n)<n ~\mbox{and}~\Re \pi_{1}(w_1)=\Re \pi_{1}(w_n)$$
for $n>1.$ Then $\Re \pi_{j_0+1}(w_n)$ is bounded, or there exists $n_1$ such that
$\Re \pi_{j_0+1}(w_{n_1}) \leqslant n_1$. Applying the same argument repeatedly
we see that there exists $n_z$ such that $$\max~\Big\{\Re
\pi_j((P^{n_z})^{-1}\circ\hdots\circ(P^2)^{-1}\circ(P^1)^{-1}(z))\Big\} \le
n_z.$$
Let $$w=(P^{n_z})^{-1}\circ\hdots\circ(P^2)^{-1}\circ(P^1)^{-1}(z)$$ then $$\Re
\pi_j(w) \notin \mbb{N} \cup \{0\} ~\mbox{for} ~1 \le j \le k$$ i.e., $w
\in \widetilde{\Delta}$ and $S(w)=z$, which is a contradiction!
\end{proof}
\begin{thm}\label{theoremA}
 Let $A,B$ be as above and $T$, $S$ the corresponding mappings on $\Omega$ and $\widetilde{\Delta}$ as mentioned
earlier. Then for a given 
$0 < \epsilon <e 1/2 $
 there exist an injective, volume preserving holomorphic map $F:\mbb{C}^k\to
\mbb{C}^k$ such that,
 \begin{enumerate}
 \item [(i)] If $\Re\pi_j(z)>\epsilon ~\mbox{for} ~1 \le j \le k$ then $z \notin
F(\mbb{C}^k).$
 \item [(ii)] $|F(z)-S(z)|<\epsilon$ for every $z\in \mbb{C}^k \setminus
B_{\epsilon/2}$ where $B_{\epsilon/2}:=\{z\in \mbb{C}^k:{\rm dist}(z;B)<\epsilon/2\}.$
 \item [(iii)] If $z \in\Omega$ such that ${\rm dist}(z;A)>\epsilon$ then $z\in
F(\mbb{C}^k)$ and $|F^{-1}(z)-T(z)|<\epsilon$.
\end{enumerate}
\end{thm}
\begin{proof}
Let $\theta_n$ be the element of ${\rm Aut}_1(\mbb{C}^k)$ which satisfies the
conditions of the Theorem \ref{thm1} with
$\lambda_1=\lambda_2=\hdots=\lambda_k=1$, $a_1=a_2=\hdots=a_k=n$ and
$\delta=\epsilon/2^{n+1}.$
Define $$\tilde{Z}^n=\Big\{z\in\mbb{C}^k:
\mbox{dist}(z;({\tilde{Y}^n})^{c})>{\epsilon}/{2^{n+1}}\Big\}$$ for every
$n\geqslant1$, where for a set $E \subset \mbb C^k, ~E^c$ denotes the complement
$\mbb C^k \setminus E.$ Then
\begin{equation}\label{a}
 |\theta_n(z)-P^n(z)|<{\epsilon}/{2^{n+1}}
\end{equation}
for all $ z \in \tilde{Z}^n$. Here $\widetilde{Y}^n$ and $P^n$ are as defined earlier. So by condition (viii) of
Theorem \ref{thm1} it follows that
\begin{equation*}
 \sum_{n=m+2}^{\infty}\|I-\theta_n\|_m < {\epsilon}/{2^{m+2}}
\end{equation*}
for $m\geqslant1.$ Then from Lemma 8.3 in \cite{DE} the sequence $F_n$ defined as
$$F_n=\theta_1 \circ \theta_2 \circ\hdots\circ\theta_n$$
converges uniformly on compact sets in $\mbb{C}^k$ to
an injective entire mapping $F$ in $\mbb{C}^k.$
\par
\medskip
\noindent{\it Claim:} The mapping $F$ so obtained is the required map.
\par
\medskip
\noindent
We will first need to prove the following auxiliary step. let
$$\widetilde{\Delta}^n=\Big\{z\in\mbb{C}^k:\mbox{dist}(z;B)>{\epsilon}/{2}-{\epsilon}/{2^{n+1}}\Big\}.$$
Note that $\mbb{C}^k\setminus B_{\epsilon/2}\subset
\widetilde{\Delta}^n \subset \widetilde{\Delta}$ and
$\widetilde{\Delta}^n \subset \tilde{Z}^n$ for every $n\geqslant 1$.
Now define maps $S_n$ on $\widetilde{\Delta}^n$ as $$S_n(z)=P^1\circ P^2\circ
\hdots\circ P^n(z),$$ for every $n$. Then for a given $z \in \mbb{C}^k \setminus
B_{\epsilon/2}$, $S_n(z)=S(z)$ for large values of $n$. We claim that:
\begin{equation}\label{ind}
|F_n(z)-S_n(z)|<{\epsilon}/{2}-{\epsilon}/{2^{n+1}}
\end{equation}
for $z\in\tilde{\Delta}^n ,~n\geqslant1.$ It follows from (\ref{a}) that  (\ref{ind}) holds when $n=1$. So
assume that (\ref{ind}) holds for some $n\geqslant1$ and we will show
that it holds for $n+1$ as well.
\medskip
\par
\noindent
\textbf{Step 1:}
$\theta_n(\widetilde{\Delta}^n) \subset \widetilde{\Delta}^{n-1}$  for
every $n \ge 1.$
\par
\medskip
\noindent
Note that $\widetilde{\Delta}^n$ can be described as:
{\small\begin{align*}
\widetilde{\Delta}^n=&\bigcap_{j=1}^{k-1}\Big\{z\in \mbb{C}^k:
l_j-1+{\epsilon}/{2}-{\epsilon}/{2^{n+1}}<\Re\pi_j(z)<
l_j-{\epsilon}/{2}+{\epsilon}/{2^{n+1}}, l_j \ge 2
,\Re\pi_j(z)<1-{\epsilon}/{2}+{\epsilon}/{2^{n+1}}\Big\}\\
&\bigcap\Big\{z\in \mbb{C}^k:
l_k-1+{\epsilon}/{2}-{\epsilon}/{2^{n+1}}<\Re\pi_k(z)<
l_k-{\epsilon}/{2}+{\epsilon}/{2^{n+1}}, l_k \ge 1
,\Re\pi_k(z)<-{\epsilon}/{2}+{\epsilon}/{2^{n+1}}\Big\}.
\end{align*}}
By Theorem \ref{thm1}, $$|\theta_n(z)-(z+(i_1,i_2,\hdots,i_k)|\leqslant
{\epsilon}/{2^{n+1}}$$ for all $z \in \widetilde{\Delta}^n$ where $i_j
\in\{0,1\}, ~\mbox{and}~1 \le j \le k.$
Then $$\Re\pi_j(z)-i_j-{\epsilon}/{2^{n+1}}\leqslant
\Re\pi_j(\theta_n(z))\leqslant \Re\pi_j(z)-i_j+{\epsilon}/{2^{n+1}}$$ for
$1 \le j \le k.$ Since $z \in \widetilde{\Delta}^n$ we have that
\begin{align*}
&l_j-1+{\epsilon}/{2}-{\epsilon}/{2^{n+1}}-i_j-{\epsilon}/{2^{n+1}}<
\Re\pi_j(\theta_n(z))<
l_j-{\epsilon}/{2}+{\epsilon}/{2^{n+1}}-i_j+{\epsilon}/{2^{n+1}}
\end{align*}
 when $l_j  \ge 2,$ or
$$\Re\pi_j(\theta_n(z))<
1-{\epsilon}/{2}+{\epsilon}/{2^{n+1}}-i_j+{\epsilon}/{2^{n+1}},
$$
where $i_j \in\{0,1\}~\mbox{and}~1 \le j \le k-1.$ Hence
\begin{align*}
l_j-1-i_j+{\epsilon}/{2}-{\epsilon}/{2^{n}}
<\Re\pi_j(\theta_n(z))<l_j-i_j-{\epsilon}/{2}+{\epsilon}/{2^{n}}
\end{align*}
where $l_j \ge 2,$ or $$\Re\pi_j(\theta_n(z))<
1-i_j-{\epsilon}/{2}+{\epsilon}/{2^{n}},$$
where $i_j \in\{0,1\}$.When $j=k$ we have
\begin{align*}
&l_k-1+{\epsilon}/{2}-{\epsilon}/{2^{n+1}}-i_k-{\epsilon}/{2^{n+1}}<
\Re\pi_k(\theta_n(z))<
l_k-{\epsilon}/{2}+{\epsilon}/{2^{n+1}}-i_k+{\epsilon}/{2^{n+1}}
\end{align*}
when $l_k  \ge 1$, or
$$\Re\pi_k(\theta_n(z))<
-{\epsilon}/{2}+{\epsilon}/{2^{n+1}}-i_k+{\epsilon}/{2^{n+1}},
$$
where $i_k \in\{0,1\}.$ Hence
\begin{align*}
l_k-1-i_k+{\epsilon}/{2}-{\epsilon}/{2^{n}}
<\Re\pi_k(\theta_n(z))<l_k-i_k-{\epsilon}/{2}+{\epsilon}/{2^{n}}
\end{align*}
when $l_k \ge 1,$ or $$\Re\pi_k(\theta_n(z))<
-i_k-{\epsilon}/{2}+{\epsilon}/{2^{n}},$$
where $i_k \in\{0,1\}.$ Thus note that by the definition of 
$\widetilde{\Delta}^{n-1}$ it follows that
$\theta_n(z)\in \widetilde{\Delta}^{n-1}$.\\
\par
\noindent
\textbf{Step 2:}
Assuming that (\ref{ind}) holds for some $n$, we will now show that it holds
for $n+1$, i.e.,
\begin{equation*}
|F_{n+1}(z)-S_{n+1}(z)|< {\epsilon}/{2}-{\epsilon}/{2^{n+2}}
\end{equation*}
for $z\in\widetilde{\Delta}^{n+1}$. Now $F_{n+1}(z)=F_n\circ\theta_{n+1}(z)$, and by Step 1,
$\theta_{n+1}(z) \in \widetilde{\Delta}^n$ if $z\in \widetilde{\Delta}^{n+1}$.
By the induction hypothesis it follows that
\begin{equation}\label{eq1}
|F_{n+1}(z)-S_n(\theta_{n+1}(z))|<{\epsilon}/{2}-{\epsilon}/{2^{n+1}}
\end{equation}
for $z\in\widetilde{\Delta}^{n+1}.$ Also $S_{n+1}(z)=S_n\circ P^{n+1}(z)$ and by Theorem \ref{thm1}
$$|\theta_{n+1}(z)-P^{n+1}(z)|<{\epsilon}/{2^{n+2}}$$ if
$z\in\widetilde{\Delta}^{n+1}.$ Since the maps $S_n$ by definition are
translation maps it follows that
\begin{equation}\label{eq2}
 |S_n(\theta_{n+1}(z))-S_{n+1}(z)|<{\epsilon}/{2^{n+2}}
 \end{equation}
for $z\in\widetilde{\Delta}^{n+1}.$ Thus from (\ref{eq1}) and (\ref{eq2}) we get
 \begin{equation*}
|F_{n+1}(z)-S_{n+1}(z)|< {\epsilon}/{2}-{\epsilon}/{2^{n+2}}
\end{equation*}
for $z\in\widetilde{\Delta}^{n+1}.$ Hence (\ref{ind}) holds for every $n\geqslant1.$
\par
\medskip
\noindent
If $z \in \mbb{C}^k \setminus B_{\epsilon/2}$ then $z\in
\widetilde{\Delta}^n$ for all $n\geqslant1$, and $S_n(z)=S(z)$ for sufficiently large
$n$, i.e., for every $n\geqslant n_z, ~n_z\in \mathbb{N}$. So
\begin{align*}
& |F_n(z)-S(z)|<{\epsilon}/{2}-{\epsilon}/{2^{n+1}}
\end{align*}
for every $n \geqslant n_z.$ This implies \[|F(z)-S(z)|<\epsilon.\]
Hence (ii) follows.
\par
\medskip
To prove (iii), define maps $T_n$ on $S_n(\tilde{\Delta}^n)$ as $T_n=S_n^{-1}$.
Let $z\in\Omega$ be such that $\mbox{dist}(z;A)\geqslant \epsilon$. Then
 \begin{align*}
 &\mbox{dist}(z;A)>\epsilon-{\epsilon}/{2^{n}}
 \end{align*}
  for every $n$, means that
 \[\mbox{dist}(T_n(z);B)>\epsilon-{\epsilon}/{2^{n}}\]
 for every $n.$
 \par
 \medskip
 \noindent
Also, for the given $z\in\Omega$ there exist $n_z \in \mathbb{N}$ such that
$T_n(z)=T(z) ~\mbox{whenever} ~n\geqslant n_z.$
 Let $n\geqslant n_z$, and choose $u \in \mbb{C}^k$ such that
$|u|\leqslant{\epsilon}/{2}-{\epsilon}/{2^{n+1}}$. Since
$$\mbox{dist}(T_n(z);\beta)>\epsilon-{\epsilon}/{2^{n}} ~\mbox{for every}
~\beta \in B$$ it follows that
 \begin{equation*}
 \mbox{dist}(T_n(z)+u,\beta)\geqslant
\mbox{dist}(T_n(z),\beta)-\mbox{dist}(T_n(z)+u,T_n(z))>
{\epsilon}/{2}-{\epsilon}/{2^{n+1}}
 \end{equation*}
whenever $\beta \in B$. So $T_n(z)+u \in \tilde{\Delta}^n$ and the pair of points $T_n(z)$, $T_n(z)+u$ lie in the same
component, say $K$ of $\tilde{\Delta}^n.$ There exist $a
\in \mbb{C}^k$ such that $S_n(u)=u-a$ for every $u\in K.$ Set
$$G_n(u)=F_n(T_n(z)+u)-z $$ for $u \in \mbb C^k$. Again, if
$|u|\leqslant{\epsilon}/{2}-{\epsilon}/{2^{n+1}}$, then
$$S_n(T_n(z)+u)=T_n(z)+u-a=S_n(T_n(z))+u=z+u.$$ Hence
 \begin{align*}
|G_n(u)-u|&=|F_n(T_n(z)+u)-z-u|\\&=|F_n(T_n(z)+u)-S_n(T_n(z)+u)|<{\epsilon}/
{2}-{\epsilon}/{2^{n+1}}.
\end{align*}
Thus it follows from Rouche's Theorem that there exist $u_0 \in
D(0;{\epsilon}/{2}-{\epsilon}/{2^{n+1}})$ such that $G_n(u_0)=0$. This
implies that $$T_n(z)+u_0=F_n^{-1}(z)$$ and so
$$|F_n^{-1}(z)-T_n(z)|=|u_0|<{\epsilon}/{2}-{\epsilon}/{2^{n+1}}.$$ Since
$T_n(z)=T(z)$ we get
$$|F_n^{-1}(z)-T(z)|<{\epsilon}/{2}-{\epsilon}/{2^{n+1}}$$ for every
$n\geqslant n_z.$ Thus $|F_n^{-1}(z)|<+\infty$ for every $n$ and by Lemma
\ref{1}, $z \in F(\mbb{C}^k)$ and $|F^{-1}(z)-T(z)|<\epsilon.$ Hence (iii) follows.
\medskip
\par
To prove (i), let $z\in\mbb{C}^k$ such that $\Re\pi_j(z)>\epsilon$ for every
$1 \le j \le k$ then by Lemma \ref{lem1}, $\mbox{dist}(z;A)\geqslant
\epsilon$, and $T_n(z)=z+(n,n,\hdots,n).$ Let
\begin{equation*}
\alpha_n =\Big\{z\in \mbb{C}^k:\min_{1 \le j \le k}\{\Re\pi_j(z)\}=n\Big\}
\end{equation*}
for every $n \ge 0.$ Then $T_n(\alpha_0)=\alpha_n$ and $\mbox{dist}(z;\alpha_0)\geqslant \epsilon$,
so $$\mbox{dist}(T_n(z);\alpha_n)\geqslant \epsilon
>\epsilon-{\epsilon}/{2^{n}}.$$ For some $n$ choose $u \in \mbb{C}^k$ such
that $|u|\leqslant{\epsilon}/{2}-{\epsilon}/{2^{n+1}}$. Since
$$\mbox{dist}(T_n(z);\beta)>\epsilon-{\epsilon}/{2^{n}}
$$ for every $\beta \in \alpha_n,$ it follows that
\begin{equation*}
\mbox{dist}(T_n(z)+u,\beta)\geqslant
\mbox{dist}(T_n(z),\beta)-\mbox{dist}(T_n(z)+u,T_n(z))>
{\epsilon}/{2}-{\epsilon}/{2^{n+1}},
 \end{equation*}
 $~\mbox{whenever} ~\beta \in
\alpha_n.$ So for $|u|\leqslant {\epsilon}/{2}-{\epsilon}/{2^{n+1}}$, the pair of points $T_n(z)$ and
$T_n(z)+u$ lie in $Y^n(1,1,\hdots,1)$, i.e., in the same component, and \[S_n(u)=u-(n,n,\hdots,n)\] for $u \in Y^n(1,1,\hdots,1).$ Set $$G_n(u)=F_n(T_n(z)+u)-z$$ for $u \in \mbb{C}^k$. Again if
$|u|\leqslant{\epsilon}/{2}-{\epsilon}/{2^{n+1}}$, then
$$S_n(T_n(z)+u)=T_n(z)+u-(n,n,\hdots,n)=S_n(T_n(z))+u=z+u.$$ Hence
 \begin{align*}
|G_n(u)-u|&=|F_n(T_n(z)+u)-z-u|\\&=|F_n(T_n(z)+u)-S_n(T_n(z)+u)|<{\epsilon}/
{2}-{\epsilon}/{2^{n+1}}.
\end{align*}
Thus it follows from Rouche's Theorem that there exist $u_0 \in
D(0;{\epsilon}/{2}-{\epsilon}/{2^{n+1}})$ such that $G_n(u_0)=0$. This
implies that $T_n(z)+u_0=F_n^{-1}(z)$, i.e.,
$$|F_n^{-1}(z)-T_n(z)|<{\epsilon}/{2}-{\epsilon}/{2^{n+1}}$$ for every
$n\in\mathbb{N}.$ Now $|T_n(z)|\to +\infty$, hence $|F_n^{-1}(z)| \to +\infty$
as $n\to\infty.$ Thus by Lemma \ref{1}, $z \notin F(\mbb{C}^k).$
\end{proof}
\hspace{5 mm}
\par
\begin{proof}[Proof of Theorem \ref{thm1.1}]
 Set $\alpha=(1+\epsilon)^{{1}/{n}}(\cos(\pi/n))-1$ and note that
$0<\alpha<\epsilon$. Choose $m \ge 1$ such that $m\geqslant
\alpha+\epsilon+2.$ Let $F$ be the injective entire mapping satisfying the
conditions of Theorem \ref{theoremA} with $\epsilon=\alpha/2m.$ Then $F$
satisfies the following properties:
 \begin{enumerate}
   \item[(i)] If $\Re\pi_j(z)\geqslant \alpha/2m $ for every $1 \le j \le k$
then $z \notin F(\mbb{C}^k).$ \\
   \item[(ii)] If  $\Re\pi_j(z)\leqslant1-\alpha/2m ~\mbox{for every}~1 \le j \le k-1 $
and $\Re\pi_k(z)\leqslant -\alpha/2m$ then $z\in F(\mbb{C}^k).$\\
   \item[(iii)] If $\Re\pi_{1}(z)\leqslant -\alpha/2m$, $\Re\pi_j(z)\leqslant
1-\alpha/2m ~\mbox{for every}  ~2 \le j \le k-1$ and  \\ $\alpha/2m\leqslant
\Re\pi_k(z)\leqslant 1-\alpha/2m$ then $z\in F(\mbb{C}^k).$ \\
   \item[(iv)] Fix an $l$ such that $1 \le l \le k-2$. If 
\begin{align*}
\alpha/2m \leqslant &\Re\pi_i(z)\leqslant 1-\alpha/2m ~\mbox{for every}~1 \le i \le l, \\
&\Re\pi_{l+1}(z)\leqslant -\alpha/2m, \\ &\Re\pi_j(z)\leqslant 1-\alpha/2m
~\mbox{for every} ~l+2 \le j \le k-1,~{\rm and}\\ \alpha/2m\leqslant
&\Re\pi_k(z)\leqslant 1-\alpha/2m
\end{align*}
 then $z\in F(\mbb{C}^k).$
 \end{enumerate}
 Property (i) follows directly from 
Theorem \ref{theoremA}(a).
 \par
 \medskip
 \noindent
 Now consider the set $\Delta(1,1,\hdots,1,0)\subset \widetilde{\Delta}$ and note that
$S_{|\Delta(1,1,\hdots,1,0)}=P^1\equiv\mbox{Identity}$. So if $z$ satisfies
property (ii) then $$z \in S(\Delta(1,1,\hdots,1,0))~\mbox{and}
~\mbox{dist}(z;(S(\Delta(1,1,\hdots,1,0)))^c) > \alpha/2m.$$ Since $A \subset
(S(\Delta(1,1,\hdots,1,0)))^c$ and dist$(z;A)>\alpha/2m$, it follows from
Theorem \ref{theoremA}(c) that, $z \in F(\mbb{C}^k).$
 \par
 \medskip
 \noindent
 Again, consider $$\Delta(1,1,\hdots,1,1)=\{z\in \mbb{C}^k:
\Re\pi_j(z)<1, ~\forall ~1 \le j \le k-1 ,~\mbox{and} ~0<\Re\pi_k(z)<1\}$$ which is a subset
of $\widetilde{\Delta}$ and
$S_{|\Delta(1,1,\hdots,1,1)}(z)=P^1(z)=z-(1,0,\hdots,0)$. So if $z$ satisfies
property (iii) then $$z \in S(\Delta(1,1,\hdots,1,1))~\mbox{and}
~\mbox{dist}(z;(S(\Delta(1,1,\hdots,1,1)))^c) > \alpha/2m.$$ Since $A \subset
(S(\Delta(1,1,\hdots,1,1)))^c$ and dist$(z;A)>\alpha/2m$, it follows from
Theorem \ref{theoremA}(c) that, $z \in F(\mbb{C}^k)$.
 \par
 \medskip
 \noindent
 Similarly if $z$ satisfies property (iv) for any $1 \le l \le k-2$
means \begin{align*}
z \in S(\Delta(\underbrace{2,2,\hdots,2}_l,1,\hdots1,1))
\end{align*}
and $$~\mbox{dist}(z;(S(\Delta(2,2,\hdots,2,1,\hdots,1,1)))^c) > \alpha/2m.$$ So ${\rm
dist}(z;A)>\alpha/2m$, and $z \in F(\mbb{C}^k)$.\\

 Define an affine map $L:\mbb C^k \to \mbb C^k$ as
 \begin{equation*}
 L(z_1,z_2,...,z_k)=(-mz_1,-mz_2,\hdots,-mz_k)+(1+{\alpha}/{2},1+{
\alpha}/{2},\hdots,1+{\alpha}/{2}).
 \end{equation*}
  Let $H:\mbb C^k \to \mbb C^k$ be given by $H(z)=L\circ F(z).$ Then $H$
satisfies the following properties:
  \begin{enumerate}
   \item[(i)$'$] If $\Re\pi_j(z)\leqslant 1 $ for every $1 \le j \le k$ then $z
\notin H(\mbb C^k).$ \\
   \item[(ii)$'$] If  $\Re\pi_k(z)\geqslant-m+1+\alpha~\mbox{for all}~1 \le j \le k-1 $ and
$\Re\pi_k(z)\geqslant 1+\alpha$ then $z\in H(\mbb C^k).$\\
   \item[(iii)$'$] If $\Re\pi_{1}(z)\geqslant 1+\alpha$, $\Re\pi_j(z)\geqslant
-m+1+\alpha~\mbox{for all} ~2 \le j \le k-1$ and $-m+1+\alpha\leqslant
\Re\pi_k(z)\leqslant 1+\alpha$ then $z\in H(\mbb C^k).$\\
   \item[(iv)$'$] Fix $l$ such that $1 \le l \le k-2$. If
\begin{align*}
-m+1+\alpha \leqslant  &\Re\pi_i(z)\leqslant 1+\alpha ~\mbox{for all} ~1 \le i \le l, \\
&\Re\pi_{l+1}(z)\geqslant 1+\alpha, \\ &\Re\pi_j(z)\geqslant -m+1+\alpha~\mbox{for all}
~l+2 \le j \le k-1~\mbox{, and} \\-m+1+\alpha\leqslant &\Re\pi_k(z)\leqslant 1+\alpha
\end{align*}
then $z\in H(\mbb C^k).$
 \end{enumerate}
 This can be checked in a straight forward manner using the properties (i)-(iv) stated above and the definition of $L(z).$
 \par
 \medskip
 \noindent
Let $\psi(z) = (z_1^n,z_2^n,\hdots,z_k^n) $
and set $G=\psi \circ H.$
\par
\medskip
\noindent
 {\it Claim}: $G$ is the required mapping.\\
  \par
  \no
  Clearly, the fibre $G^{-1}(z)$ possesses at most $n^k$ elements. Now if $z\in \mbb
C^k$ then by property (i)$'$
  \begin{equation*}
  \max\Big\{\Re(\pi_1\circ H)(z),\Re(\pi_2\circ H)(z),\hdots,\Re(\pi_k\circ H)(z)\Big\}> 1
  \end{equation*}
  so that
  \begin{align*}
  &\max \Big\{|\Re(\pi_1\circ H)(z)|,|\Re(\pi_2\circ H)(z)|,\hdots,|\Re(\pi_k\circ
H)(z)|\Big\} > 1 ,
\end{align*}
and this gives
\begin{align*}
  &\max \Big\{|\Re(\pi_1\circ G)(z)|,|\Re(\pi_2\circ G)(z)|,\hdots,|\Re(\pi_k\circ
G)(z)|\Big\} > 1.
  \end{align*}
Hence $G(\mbb C^k)\cap \overline{D(0;1)}=\emptyset.$ \\
 \par
 Let $u=(\rho_1e^{i\theta_1},\rho_2e^{i\theta_2},\hdots,\rho_pe^{i\theta_k})$
with $|\theta_j|\leqslant \pi$ and $\rho_j\geqslant 0$  for every $1 \le j \le k$,
be a point in $\mbb C^k \setminus D(0;1+\epsilon).$\\
 \par
 \noindent
\textit{Case 1:} Suppose $\rho_k\geqslant 1+\epsilon$. Set
$v_k=\rho_k^{1/n}\cos(\theta_k/n)$. Then $$\Re
v_k\geqslant(1+\epsilon)^{1/n}\cos(\pi/n)=1+\alpha.$$ Let $\xi_k$ be the $n$-th
root of $e^{i\theta_j}$ such that $\Re\xi_j\geqslant 0$ and set
$v_j=\rho_j^{1/n}\xi_j$ for every $1 \le j \le k-1.$ So $\Re
v_j\geqslant0>-m+1+\alpha$  for every $1 \le j \le k-1.$
 Hence by property (ii)$'$, $v=(v_1,v_2,\hdots,v_k)\in H(\mbb C^k)$ and
$u=\psi(v)\in G(\mbb C^k).$\\
 \par
 \noindent
\textit{Case 2:} Suppose $\rho_1\geqslant 1+\epsilon$ and $\rho_k<1+\epsilon$.
Set $v_1=\rho_1^{1/n}\cos(\theta_1/n)$. Then $$\Re
v_1\geqslant(1+\epsilon)^{1/n}\cos(\pi/n)=1+\alpha.$$ Let $\xi_j$ be the $n$-th
root of $e^{i\theta_j}$ such that $\Re \xi_j\geqslant 0$ and set
$v_j=\rho_j^{1/n}\xi_j$ for every $2 \le j \le k-1.$ So $\Re
v_j\geqslant0>-m+1+\alpha$  for every $2 \le j \le k-1.$ Let $\xi_k$ be the $n$-th
root of $e^{i\theta_k}$ such that $\Re \xi_k\leqslant 0$ and set
$v_k=\rho_k^{1/n}\xi_k.$ So
$$-m+1+\alpha\leqslant-(1+\epsilon)<-(1+\epsilon)^{1/n}<\Re
v_k\leqslant0<1+\alpha.$$ Hence by property (iii)$'$,
$v=(v_1,v_2,\hdots,v_k)\in H(\mbb C^k)$ and $u=\psi(v)\in G(\mbb C^k).$\\
 \par
 \noindent
\textit{Case 3:} Fix $l \in \{1,2,\hdots,k-2\}$ and suppose that $\rho_{l+1}\geqslant
1+\epsilon$, $\rho_k<1+\epsilon$ and $\rho_p<1+\epsilon$ for every
$1 \le p \le l$. Set $v_{l+1}=\rho_{l+1}^{1/n}\cos(\theta_{l+1}/n)$. Then
 $$\Re v_{l+1}\geqslant(1+\epsilon)^{1/n}\cos(\pi/n)=1+\alpha.$$ Let $\xi_j$ be
the $n$-th root of $e^{i\theta_j}$ such that $\Re~\xi_j\geqslant 0$ and set
$v_j=\rho_j^{1/n}\xi_j$ for every $l+2 \le j \le k-1.$
  \par
  \medskip
  \noindent
So $\Re v_j\geqslant0>-m+1+\alpha$  for every $l+2 \le j \le k$. Let $\xi_p$
be the $n$-th root of $e^{i\theta_p}$ such that $\Re \xi_p\leqslant 0$. Set
$v_p=\rho_p^{1/n}\xi_p$ for every $1 \le p \le l$. Also let $\xi_k$ be the $n$-th
root of $e^{i\theta_k}$ such that $\Re \xi_k\leqslant 0$. Set
$v_k=\rho_k^{1/n}\xi_k.$ Thus
$$-m+1+\alpha\leqslant-(1+\epsilon)<-(1+\epsilon)^{1/n}<\Re
v_p\leqslant0<1+\alpha$$ for every $1 \le p \le l.$ Also similarly
$$-m+1+\alpha<\Re v_k\leqslant0<1+\alpha.$$
Hence by property (iv), $v=(v_1,v_2,\hdots,v_k)\in H(\mbb C^k)$ and
$u=\psi(v)\in G(\mbb C^k).$
\par
\medskip
\noindent
Therefore $\mbb C^k \setminus D(0;1+\epsilon)\subset G(\mbb C^k)$ and hence $G$ is the
required napping.
\end{proof}

\section{Proof of Theorem \ref{thm1.2}}
\noindent
In this section we will again use transcendental shift-like maps to construct an injective holomorphic map from $\mbb C^k$
to $\mbb C^k$ whose range satisfies the properties mentioned in Theorem \ref{thm1.2}. The range is then a
Fatou-Bieberbach domain in $\mbb C^k, ~k \ge 2$ with these properties.
\par
\medskip
\noindent
Recall the definition of the vertical strips $V_{\pm l}^a, V_0^a, V_{\pm (K+1)}^a$ as given before the statement of Theorem \ref{thm1.2}.
Note that the definition depends upon the parameters $a$ and $K$. First we construct an entire function on $\mbb C$ that can
be approximated with suitably chosen constants in these strips. This is used to construct a shift-like map, hence an element in
${\rm Aut}_1(\mbb C^k), k \ge 2.$ Call this Step 0. In Step 1, the parameter of definition of these strips is changed from $a$
to $a+2$, and similarly as before we use this entire function to construct an element of ${\rm Aut}_1(\mbb C^k).$ Proceeding
inductively in this way we create at Step $n$ (the parameter value is $a+2n$ now) an entire function on $\mbb C$ with prescibed
behaviour in $V_{\pm l}^{a+2n}, V_0^{a+2n}, V_{\pm (K+1)}^{a+2n}$ and hence an element in ${\rm Aut}_1(\mbb C^k).$ For every $n \ge 0$
we compose the $n$ elements in ${\rm Aut}_1(\mbb C^k)$ so obtained. The limit of the sequence of compositions is of interest to us.
The main job here is to define the constants appropriately on the strips at every step so that the sequence of compositions satisfies
some desired properties.

\begin{proof}[Proof of Theorem \ref{thm1.2}]
Define the following sets:

\begin{itemize}
\item $\bar{D}^a_l = \bar{D}^0_l(a) \cup \bar{D}^1_{\pm l}(a) \cup \bar{D}^{\pm 2}_l(a) \cup \ldots$, where for
each $n \in \mbb Z$ and $1  \le \vert l \vert \le K + 1$, $\bar{D}^n_l(a)$ is a disc of
radius $1/4$ centered at the point $(a + (4l -3)/2,
2n) \in \mbb C$, and
\medskip
\item $\bar{E}^a_l = \bar{E}^0_l(a) \cup \bar{E}^{\pm 1}_l(a) \cup \bar{E}^{\pm 2}_l(a) \cup \ldots$, where for
each $n \in \mbb Z$ and $1 \le \vert l \vert \le K + 1$, $\bar{E}^n_l(a)$ is a disc of
radius $1/4$ centered at the point $(a + (4l -3)/2,
2n+1) \in \mbb C$.
\end{itemize}
These are slightly thicker discs than the corresponding ones defined in Section 1.
\par \medskip \noindent
For any $a>0$ there exist entire functions $f_{a+i}, ~i \ge 0$, satisfying the following
properties:
\begin{align*}
 \mbox{(i)} \hspace{10 mm} &|f_{a+i}|<2^{-(4+i)}~\mbox{on}~
V_0^{a+i}~\mbox{for every}~ i \ge 0 ,\\
 \mbox{(ii)} \hspace{10 mm} &|f_{a+i}-M|<1~\mbox{on}~ V_l^{a+i}~\mbox{for every}~
i \ge 0 , ~\mbox{and}~1 \le l \le K+1 ,\\
 \mbox{(iii)} \hspace{10 mm} &|f_{a+i}+M|<1~\mbox{on}~ V_{-l}^{a+i}~\mbox{for
every}~ i \ge 0 , ~\mbox{and}~1 \le l \le K+1,\\
 \mbox{(iv)} \hspace{10 mm} &|f_{a+i}-M|<1~\mbox{on}~ \bar{D}_l^{a+i}~\mbox{for
every}~ i \ge 0 , ~\mbox{and}~1 \le |l| \le K+1,\\
\mbox{(v)} \hspace{10 mm} &|f_{a+i}-M|<2^{-(4+i)}~\mbox{on}~
\bar{E}_l^{a+i}~\mbox{for every}~ i \ge 0
~\mbox{and}~1 \le |l| \le K+1,\\
 \end{align*}
where $M$ is such that $M>2a+4K+5.$
\medskip
\par
\noindent The existence of these functions is assured by the Arakelian--Gauthier Theorem \cite{AG}, used similarly as in the proof of Theorem \ref{thm1a}.
Now define $\theta_{a+i}: \mbb C^k \to \mbb C^k$, $\theta_{a+i}(z_1,z_2,\hdots,z_k)=(u_1,u_2,\hdots,u_k)$ where
\begin{align*}
u_1 &= z_1 + f_{a+i}(z_2),\\
u_2 &= z_2 + f_{a+i}(z_3), \\
&\vdots\\
u_k &= z_k + f_{a+i}(u_1),
\end{align*}
for every $i \ge 1$ and let
$$F_n=\theta_a^{-1}\circ \theta_{a+2}^{-1}\circ\hdots \circ
\theta_{a+2(n-1)}^{-1}$$
for every $n \ge 1.$
\par
\medskip
\noindent
Note that 
\begin{align*}\theta_{a+2i}^{-1}(u_1,u_2, \hdots,u_k)&=(z_1,z_2,\hdots,z_k) \\
&=(u_1-f_{a+2i}(z_2), u_2-f_{a+2i}(z_3),\hdots,u_k-f_{a+2i}(u_1)).
\end{align*}

\noindent
{\it Claim:} $F_n$ converges uniformly on compact subsets of $\mbb C^k$ to $F: \mbb C^k \to \mbb C^k.$
\par
\medskip
\noindent
By Lemma 8.3 of \cite{DE} it suffices to show that
\[ \sum_{i=1}^\infty \| I-\theta_{a+2i}^{-1}\|_m< \infty\] for any $m>0$.
\par
\medskip
\noindent
Note that
\[\| I-\theta_{a+2i}^{-1}\|_m \le \max\Big\{ \|f_{a+2i}(u_1)\|_m,
\|f_{a+2i}(z_2)\|_m,\hdots, \|f_{a+2i}(z_k)\|_m\Big\}.\]
\medskip
\noindent
Pick $u=(u_1,u_2,\hdots,u_3) \in \mbb C^k$ with $|u|\le m$. Then there exists sufficiently large $i$ such that $a+2i>2m$,
which means that $u_j \in V_0^{a+2i}$ for all $1 \le j \le k$ and hence $|f_{a+2i}(u_j)|<{2^{-(4+2i)}}$. Also
since $|u_k|<m$,
\begin{align*}
 |z_k|<|u_k|+|f_{a+2i}(u_1)|<m+2^{-(2i+4)}<a+2i,
\end{align*}
i.e., $z_k \in V_0^{a+2i}$ and so
\[|f_{a+2i}(z_k)|<{2^{-(4+2i)}}.\]
Since
\[|z_{k-1}|<|u_{k-1}|+|f_{a+2i}(z_k)|< m+2^{-(2i+4)}<a+2i,\]
 $z_{k-1} \in V_0^{a+2i}$ and
\[|f_{a+2i}(z_{k-1})|<{2^{-(4+2i)}}.\]
By applying the same argument for $k-2,k-3, \hdots, 1$ we have
\[|f_{a+2i}(z_j)|<{2^{-(4+2i)}}\]
for every $1\le j \le k.$ Hence
\[\| I-\theta_{a+2i}^{-1}\|_m < {2^{-(2i+4)}}  \]
whenever $a+2i > m.$ Thus $F_n$ converges to some $F$ uniformly on compact subsets of $\mbb C^k.$
\par
\medskip
\noindent
 {\it Claim:} $F$ satisfies all the properties of Theorem \ref{thm1.2}.
 \par
 \medskip
 \noindent
{\bf Step 1:} We first show that
\[|F_n^{-1}(z)| \to \infty ~\mbox{if}~ z=(z_1,z_2,\hdots,z_k) \in V_{l_1}^a
\times V_{l_2}^a \times \hdots \times V_{l_k}^a\]
 whenever $1 \le |l_j| \le K+1.$
 \par
 \medskip
 \noindent
 {\it Claim:}
 \[\theta_{a}\Big(V_{l_1}^a \times V_{l_2}^a \times
\hdots \times V_{l_k}^a\Big) \subset V_{K+1}^{a+2} \times V_{K+1}^{a+2} \times \hdots \times
V_{K+1}^{a+2} \]
for every $1 \le l_j \le K+1.$
\par
\medskip
\noindent
Pick
 $z \in V_{l_1}^a \times V_{l_2}^a \times
\hdots \times V_{l_k}^a.$ Then $\Re z_j> a+1$ for every $1 \le j \le k$ and therefore $\Re f_a(z_j)> M-1.$
 Since  $$\Re u_j=\Re z_j+ \Re f_a(z_{j+1}),$$
 we have
 \[\Re u_j> a+3+2K=(a+2)+(2K+1)\]
i.e., $u_j \in V_{k+1}^{a+2}$
 for every $1 \le j \le k-1.$ When $j=k,$
 $$\Re u_k=\Re z_k+ \Re f_a(u_1).$$
 Since $u_1 \in V_{K+1}^{a+2}\subset V_{K+1}^{a}$, $\Re f_a(u_1)> M-1.$ Thus by a similar argument as above
 we have that $u_k \in V_{k+1}^{a+2}.$ Hence the claim.
 \par
 \medskip
 \noindent
Pick $z \in V_{K+1}^{a+2i} \times V_{K+1}^{a+2i} \times \hdots \times
V_{K+1}^{a+2i}$ for some $i \ge 1$. Then $\Re z_j>a+2i+1$ for every $1 \le j \le k$ and therefore $\Re f_{a+2i}(z_j)> M-1$. So if we apply the same
arguments as used before it follows that
\begin{equation}\label{2a}
\theta_{a+2i}\Big(V_{K+1}^{a+2i} \times V_{K+1}^{a+2i} \times \hdots \times
V_{K+1}^{a+2i}\Big) \subset V_{K+1}^{a+2(i+1)} \times V_{K+1}^{a+2(i+1)}\times
\hdots \times V_{K+1}^{a+2(i+1)}
\end{equation}
for every $i \ge 1.$
Thus $$|F_n^{-1}(z)|=|\theta_{a+2(n-1)}\circ\hdots\circ\theta_a(z)|\to
\infty$$ if $z \in V_{l_1}^a \times V_{l_2}^a \times \hdots \times V_{l_k}^a ~,
~1\le l_j \le K+1 .$
\par
\medskip
\noindent
 Similarly for $z \in V_{-l_1}^a \times V_{-l_2}^a \times \hdots \times
V_{-l_k}^a ~, ~1\le l_j \le K+1,$ we use the fact that
\[\Re f_a(z_i)< -(a+1)\] i.e., \[\Re f_a(z_j)<-M+1 ~\mbox{for every}
~1 \le j \le k\]
 to obtain
 \[ \theta_{a}\Big(V_{-l_1}^a \times V_{-l_2}^a \times \hdots \times V_{-l_k}^a\Big)
\subset V_{-(K+1)}^{a+2} \times V_{-(K+1)}^{a+2} \times \hdots \times V_{-(k+1)}^{a+2}\]
and also
\begin{equation*}
\theta_{a+2i}\Big(V_{-(K+1)}^{a+2i} \times V_{-(K+1)}^{a+2i} \times \hdots
\times V_{-(K+1)}^{a+2i}\Big) \subset V_{-(K+1)}^{a+2(i+1)} \times
V_{-(K+1)}^{a+2(i+1)}\times \hdots \times V_{-(K+1)}^{a+2(i+1)}
 \end{equation*}
 for every $i \ge 1.$
 Hence $$|F_n^{-1}(z)|=|\theta_{a+2(n-1)}\circ\hdots\circ\theta_a(z)|\to
\infty$$ if $z \in V_{-l_1}^a \times V_{-l_2}^a \times \hdots \times V_{-l_k}^a
~, ~1 \le l_j \le K+1.$
 \par
 \medskip
 \noindent
 {\bf Step 2:} We show that
 \[|F_n^{-1}(z)| \to \infty ~\mbox{if}~z=(z_1,z_2,\hdots,z_k) \in D_{l_1}^a
\times D_{l_2}^a \times \hdots \times D_{l_k}^a\]
 whenever $1 \le |l_j| \le K+1.$
 \par
 \medskip
 \noindent
 {\it Claim:}
 \[\theta_{a}\Big(D_{l_1}^a \times D_{l_2}^a \times
\hdots \times D_{l_k}^a\Big) \subset V_{K+1}^{a+2} \times V_{K+1}^{a+2} \times \hdots \times
V_{K+1}^{a+2} \]
for every $ 1 \le |l_j| \le K+1.$
\par
\medskip
\noindent
To see this note that if $z_j \in D_{l_j}^a$ then $$\Re z_j>-(a+2K+1),$$ for every $1 \le j \le k$. See Figure 1. Now
since,  $$\Re f_a(z_j)> M-1>2a+4K+4,$$
by similar arguments as before we get that
 \[ \Re u_j> -(a+2K+1)+2a+4K+4=(a+2)+(2K+1)\]
for every $~1 \le j \le k.$
Thus the claim.
\par
\medskip
\noindent
Now from (\ref{2a}) it follows
$|F_n^{-1}(z)|\to \infty$ if $z \in D_{l_1}^a \times D_{l_2}^a \times \hdots
\times D_{l_k}^a $ where $1 \le |l_j| \le K+1.$
 \par
 \medskip
 \noindent
 {\bf Step 3:} We show that
 \[|F_n^{-1}(z)| < \infty ~\mbox{if}~ z=(z_1,z_2,\hdots,z_k) \in E_{l_1}^a
\times E_{l_2}^a \times \hdots \times E_{l_k}^a.\]
 If $z_j \in E_{l_j}^a$ then $z_j \in E_{l_j}^{n_j}(a)$ for some $n_j \ge 0.$
 Let $u_j^n = \pi_j(F_n^{-1}(z))$ for $n \ge 0.$ Then
 \begin{align*}
 u_j^n=&u_j^{n-1}+f_{a+2(n-1)}(u_{j+1}^{n-1})
 \end{align*}
for every
$1 \le j \le k-1$, and
\begin{align*}
 u_k^n=&u_k^{n-1}+f_{a+2(n-1)}(u_{1}^{n}).
\end{align*}
Thus
\begin{align*}
    |u_j^1 \pm (a+{4|l_j|-3})/{2}+i(2n_j+1)| &\le |z_j\pm
(a+{4|l_j|-3})/{2}+i(2n_j+1)|+|f_{a}(z_{j+1})| \\ &< {1}/{8}+{1}/{16} <
{1}/{4},
 \end{align*}
  i.e., $u_j^1 \in \bar{E}_{l_j}^{n_j}(a)$ for every $1 \le j \le k-1,$ and
 \begin{align*}
   |u_k^1 \pm (a+{4|l_k|-3})/{2}+i(2n_k+1)| &\le |z_k\pm
(a+{4|l_k|-3})/{2}+i(2n_k+1)|+|f_{a}(u_{1})| \\ &< {1}/{8}+{1}/{16} <
{1}/{4},
 \end{align*}
  i.e., $u_k^1 \in \bar{E}_{l_k}^{n_k}(a).$
  \medskip
  \par
  \noindent
  {\it Case 1:} If $\max \{|l_1|,|l_2|, \hdots ,|l_k|\}=1$, then from 
the equations above we have
  \begin{align*}
  \Re u_j^1 &\le a+({4|l_j|-3})/{2}+{1}/{4} < a+1, \\
  \Re u_j^1 &\ge -(a+({4|l_j|-3})/{2}+{1}/{4})>-(a+1)
  \end{align*}
  for every $1 \le j \le k.$ Hence
  \[\theta_n\Big(E_{l_1}^{n_1}(a)\times \hdots \times E_{l_k}^{n_k}(a) \Big) \subset
V_0^{a+2}\times \hdots \times V_0^{a+2}.\]
\noindent
We make the following induction statement:
 \begin{equation}
 -\bigg(a+1 + \sum_{m=1}^{n-1} {2^{-(2m+4)}}\bigg)\le \Re u_j^n \le \bigg(a+1 +
\sum_{m=1}^{n-1} {2^{-(2m+4)}}\bigg)
 \end{equation}
for all $n \ge 1.$ The induction statement is true when $n=1.$ Assume the statement is true for
some $n \ge 1$, i.e., $u_j^n \in V_0^{a+2n}.$ Then
\begin{align}\label{2b}
 &|f_{a+2n}(u_j^n)|< {2^{-(2n+4)}}
\end{align}
 which gives
 \[ -{2^{-(n+4)}}< \Re f_{a+2n}(u_j^n)<{2^{-(2n+4)}}\]
 and so
 \[-\bigg(a+1 + \sum_{m=1}^n {2^{-(2m+4)}}\bigg)\le \Re u_j^{n+1} \le \bigg(a+1
+ \sum_{m=1}^n {2^{-(2m+4)}}\bigg)
\]
This completes the induction argument. Thus we get
 \[ -(a+2) \le \Re u_j^n \le a+2  \]
 for every $n \ge 0.$
 \par
 \medskip
 \noindent
Also, by arguing as in (\ref{2b}) we have that
 \[ |u_j^n| \le |u_j^0|+\sum_{m=1}^{n-1} {2^{-(m+4)}}\]
for any $n \ge 0.$ Hence the sequence $| F_n^{-1}(z)|$ is bounded which implies that $z \in F(\mbb
C^k).$
 \par
 \medskip
 \noindent
 {\it Case 2:} If $\max \{|l_1|,|l_2|, \hdots, |l_k|\}>1,$ 
 let $N=\max \{|l_j|:1 \le j \le k \}.$
 \medskip
 \par
 \noindent
 We make the following induction statement:
 \begin{align}\label{2c}
 |u_j^n \pm (a+(4|l_j|-3)/{2}+i(2n_j+1))|\le {1}/{8}+
\sum_{m=0}^{n-1}{2^{-(4+2m)}}
 \end{align}
\noindent
for every $n \ge 1.$
Observe that when $n=1$, since \[|f_{a}(u_j^0)| \le 2^{-4},\hspace{3mm} |f_{a}(u_1^1)| \le 2^{-4}\] for every $1 \le j \le k$ the induction statement (\ref{2c}) is true. So if $|l_j|>1$ then
$$u_j^1 \in \bar{E}_{\pm(|l_j|-1)}^{n_j}(a+2)$$ otherwise, as in Case 1
\[u_j^1 \in V_0^{a+2}.\]
Now assume the induction statement holds for some $n$. Then if $n < N$
 $$u_j^n \in \bar{E}_{\pm(|l_j|-n)}^{n_j}(a+2n)$$
 for some $1 \le j \le k$ otherwise, i.e., if $n \ge N$ then 
 \[u_j^n \in V_0^{a+2n}\] for every $1 \le j \le k.$ 
 In either case
 $|f_{a+2n}(u_j^n)|<{2^{-(n+4)}}$
 which implies that $$|u_j^{n+1} \pm (a+(4|l_j|-3)/{2}+i(2n_j+1))|\le {1}/{8}+
\sum_{m=0}^{n}{2^{-(4+2m)}}$$ whenever $1 \le j \le k-1.$ Now arguing in the
same way for $u_1^{n+1}$ we see that (\ref{2c}) holds for $u_k^{n+1}$ as well.
Thus $|F_n^{-1}(z)|$ is bounded and hence $z \in F(\mbb C^k).$

\end{proof}
\section{Polynomial Shift-like maps}
\noindent
In this section we prove Theorems \ref{thm1.3}, \ref{thm1.4} and \ref{thm1.5}.
\par
\medskip
\noindent
Consider finitely many polynomial shift-like maps of type 1 in $\mbb C^k$ of the form
\[ F_l(z_1,z_2,\hdots, z_k)=(z_2,z_3,\hdots,\alpha_l z_1+p_l(z_k))\]
where $1 \le l \le m$ and
$p_l$ is a polynomial with degree $d_l \ge 2$ and the constants $\alpha_l \neq 0.$
Let 
\begin{align}\label{3}
F=F_m\circ F_{m-1}\circ \hdots \circ F_1
\end{align} 
The degree of $F$ is $d=d_1d_2 \hdots d_m.$ 
\par
\medskip
\noindent
Recall the sets $K^{\pm}, K,V^{\pm},V$ for the map $F$ as defined in Section 1. From \cite{BP} we have the following basic result:
\begin{lem}
$K^{+} \subset V\cup V^{-}$ and $K^{-} \subset V \cup V^{+}.$ Moreover
for every $1 \le j \le m $, let $G_j=F_j \circ F_{j-1} \circ \hdots \circ F_1.$ Then $G_j(K^{+})\subset V \cup V^{-}$ and $G_j(K^-) \subset V \cup V^{+}.$
\end{lem}
\no
Using this we have the following refined estimates.
\begin{lem}\label{lem3.1}
For an arbitrary $\epsilon >0$ there exists $M>0$ such that 
\begin{enumerate}
\item[(i)]If $(z_1,z_2,\hdots,z_k) \in K^-$ then
\begin{align*}
(1- \epsilon)|z_{k-1}|^{d_{[m]}}-M &\le |z_k| \le (1+ \epsilon)|z_{k-1}|^{d_{[m]}}+M \\
(1- \epsilon)|z_{i-1}|^{d_{[m-k+i]}}-M &\le |z_i| \le (1+ \epsilon)|z_{i-1}|^{d_{[m-k+i]}}+M
\end{align*}
for every $2 \le i \le k-1.$
\item[(ii)]If $(z_1,_2,\hdots,z_k) \in G_j(K^-)$, when $1 \le j\le m-1$ then
\begin{align*}
(1- \epsilon)|z_{k-1}|^{d_{[j]}}-M &\le |z_k| \le (1+ \epsilon)|z_{k-1}|^{d_{[j]}}+M \\
(1- \epsilon)|z_{i-1}|^{d_{[j-k+i]}}-M &\le |z_i| \le (1+ \epsilon)|z_{i-1}|^{d_{[j-k+i]}}+M
\end{align*}
for every $2 \le i \le k-1.$
\end{enumerate}
\end{lem}
\begin{proof}
Let $(z_1,z_2, \hdots, z_k) \in G_j(K^-)$ for some $1 \le j \le m.$ Then there exists 
$(w_1,w_2,\hdots,w_k) \in G_{j-1}(K^-)$ such that
\begin{align*}
(z_1,z_2,\hdots,z_k)&=F_j(w_1,w_2,\hdots,w_k)\\
&=(w_1,w_2,\hdots,w_p,\alpha_j w_1+p_j(w_k))
\end{align*} 
Recall that the filtration $V, V^{\pm}$ is determined by a chosen $R>0.$
Now for sufficiently large $R>0$ there exists $M>0$ such that
\begin{align*}
 -M+(1-\epsilon)|w_k|^{d_{[j]}}\le |\alpha_j+p_j(w_k)| \le (1+\epsilon)|w_k|^{d_{[j]}}+M, ~\mbox{and}
\end{align*}
\begin{align}\label{3c}
 -M+(1-\epsilon)|z_{k-1}|^{d_{[j]}}\le| z_k| \le (1+\epsilon)|z_{k-1}|^{d_{[j]}}+M 
\end{align}
for every $1 \le j \le m.$ This gives (ii).
\par\medskip\noindent
Therefore for a given $i$ with $2 \le i \le k-1$, if $(z_1,z_2, \hdots, z_k) \in G_j(K^-)$ for some $1 \le j \le m$ there exists 
$(w_1,w_2,\hdots,w_k) \in G_{[j-k+i]}(K^-)$ such that
\begin{align*}
 \underbrace{F_{[j]}\circ F_{[j-1]}\circ \hdots \circ F_{[j-k+i+1]}}_{(k-i)}(w_1,w_2,\hdots,w_k)=(z_1,z_2,\hdots,z_k).
\end{align*}
By definition this means that
\[(\underbrace{w_{k-i+1},\hdots,w_k}_i,\tilde{w}_{i+1},\hdots,\tilde{w}_k)=(z_1,z_2,\hdots,z_k).\]
Hence $z_i=w_k $ and $z_{i-1}=w_{k-1}.$ Now  by (\ref{3c})
\begin{align}\label{3f}
M+(1+\epsilon)|z_{i-1}|^{d_{[j-k+i]}}\le| z_i| \le (1+\epsilon)|z_{i-1}|^{d_{[j-k+i]}}+M
\end{align} 
for every $1 \le j \le m.$
\par\medskip\noindent
Observe that $K^-$ is invariant under $F$, i.e., $K^-=G_m(K^-).$ Now by substituting $j=m$ in (\ref{3c}) and  (\ref{3f}), we obtain (i).
\end{proof}
\begin{proof}[Proof of Theorem \ref{thm1.3} ]
Pick any $z=(\ti{z}_1^0,\ti{z}^0_2,\hdots,\ti{z}^0_k) \in K^-$. Then 
\begin{align*}
(\ti{z}_1^0,\ti{z}_2^0,\hdots,\ti{z}_k^0)&\xrightarrow{F_1}(\ti{z}_2^0,\ti{z}_3^0,\hdots,\ti{z}_k^0,\ti{z}_k^1)\xrightarrow{F_2} \hdots \\
&\xrightarrow{F_m}(\ti{z}_k^{-k+m+1},\ti{z}_k^{-k+m+2},\hdots,\ti{z}^k_k)
\end{align*}
where we define $\ti{z}_k^{-j}=\ti{z}_{k-j}^0$ for $0\le j\le k-1.$
\par
\medskip
\noindent
By Lemma \ref{lem3.1} it follows that
\begin{align*}
|\ti{z}_k^i| \le (1+\epsilon)\max \Big\{|\ti{z}_k^{i-1}|^{d_{[i]}}, M^{d_{[i]}}\Big\}
\end{align*}
for $1 \le i \le m.$ Thus
\begin{align*}
|\ti{z}_k^m| &\le (1+\epsilon)\max\Big\{(1+\epsilon)\max\Big\{\hdots \max\Big\{|\ti{z}_k^0|,M^{d_1}\Big\}^{d_2},\hdots\Big\}^{d_m},M^{d_m}\Big\} \\
& \le (1+\epsilon)^{1+d_{[m]}+d_{[m]}d_{[m-1}+\hdots+d_{[m]}d_{[m-1]}\hdots d_{[2]}} \max\Big\{|\ti{z}_k^0|^d,M^d\Big\} \\
&\le(1+\epsilon)^{1+D_k^1+D_k^2+\hdots +D_k^{m-1}}\max\Big\{|\ti{z}_k^0|^d,M^d\Big\} \\ & \le C_k(\epsilon)\max\Big\{|\ti{z}_k^0|^d,M^d\Big\}.
\end{align*}
For any $j$, with $2 \le j\le k-1$ by using the facts 
\begin{align*}
|\ti{z}_j^m|&=|\ti{z}_k^{-k+m+j}| \\
& \le (1+\epsilon)^{1+d_{[-k+m+j]}+\hdots+d_{[-k+m+j]}\hdots d_{[2]}} \max\Big\{|\ti{z}_k^0|^{d_{[-k+m+j]}\hdots d_{[1]}},M^{d_{[-k+m+j]}\hdots d_{[1]}}\Big\}
\end{align*}
and 
\begin{align*}
|\ti{z}_{i+1}^0| \le (1+\epsilon)\max\Big \{|\ti{z}_i^0|^{d_{[m-k+i+1]}},M^{d_{[m-k+i+1]}}\Big\}
\end{align*}
for every $1 \le i \le k-1$, it can be concluded that
\begin{align*}
|\ti{z}_j^m| &\le (1+\epsilon)^{1+D_j^1+D_j^2+\hdots +D_j^{m-1}}\max\Big\{|\ti{z}_j^0|^d,M^d\Big\} \\
&\le C_j(\epsilon)\max\Big\{|\ti{z}_j^0|^d,M^d\Big\}.
\end{align*}
Now note that if $$\ti{z}_j^0=\pi_j \circ F^n(z)=z_j^n$$ then
\[|\ti{z}_j^m|=\pi_j \circ F(z_1,z_2,\hdots,z_k)=z_j^{n+1} \]
where $n \ge 0$ and $1 \le j \le k.$
Thus applying the above argument for any $n \ge 1$ we have
\begin{align*}
|z_{j}^{n+1}| &\le C_j(\epsilon)\max\Big\{ |z_j^n|^d,M^d\Big\} \\
& \le  C_j(\epsilon)^{1+d+\hdots+d^n}\max\Big\{ |z_j^0|^{d^{n+1}},M^{d^{n+1}}\Big\} \\
& \le C_j(\epsilon)^{d^{n+1}}\max\Big\{ |z_j^0|^{d^{n+1}},M^{d^{n+1}}\Big\}
\end{align*}
which is the first estimate.
\par \medskip \noindent
To prove the second inequality, i.e.,
\[
\max \Big\{ \vert z_j^n \vert, C_j(-\ep)^{d^n} M^{d^n} \Big\} \ge C_j(-\ep)^{d^n}
\vert z_j^0 \vert^{d^n}.
\]
again Lemma \ref{lem3.1} is applied to obtain
\[\max\Big\{|\ti{z}_j^0|,(1-\ep)M\Big\} \ge (1-\ep)|\ti{z}_{j-1}^0|^{d_{[m-k+j]}}\]
for every $2 \le j\le k.$
\par\medskip\noindent
Now repeating the steps as before the required inequality is achieved.
\end{proof}
\par \medskip \noindent
The next statement concerns the parametrization of an unstable manifold of an automorphism
with a $(k-1,1)$ saddle point and the conjugacy of the restriction of the automorphism 
to the unstable manifold to its linear part. It is a consequence of Sternberg's theorem (See \cite{SS} and \cite{Sb}). The proof given below is an adaptation of the proof in $\mbb C^2$ given in \cite{MNTU}.
\begin{lem}\label{3lem(a)}
Let $F \in {\rm Aut}(\mbb C^k)$ with a saddle point $a \in \mbb C^k$ of type $(k-1,1).$ Then there exists a holomorphic map $H: \mbb C \to \mbb C^k$ such that $H(\mbb C)=W^u(a)$ and $F\circ H(t)=H(\la t)$ where $W^u(a)$ is the unstable manifold at $a \in \mbb C^k$ and $\la$ is the eigenvalue of $DF(a)$ such that $|\la|>1.$
\end{lem}
\begin{proof}
By a suitable conjugation we can assume the following properties on the map $F:$
\begin{enumerate}
\item[(i)] $a=0$ is the required saddle point.
\item[(ii)] $DF(0)$ is an upper triangular matrix with $\la_k=\la$ and $|\la_i| \le 1$ for every $1 \le i \le k-1.$
\end{enumerate}
Since $DF(0)$ is an upper triangular matrix at the origin, it has the form 
\begin{align}\label{3a}
F(z)=(\la_1 z_1+f_1(z_1,z_2,\hdots,z_k),\la_2 z_2+f_2(z_1,z_2,\hdots,z_k), \hdots, \la_k z_k +f_k(z_1,z_2,\hdots,z_k))
\end{align} 
where the $f_i$'s are polynomials in $\mbb C^k$ with the property that deg $z_k \ge 2$ in each $f_i.$ 
\par
\medskip
\noindent
Since $|\la|> 1$, choose $s$ such that \[|\la| < s < |\la|^2.\]
Let $A= s^{-1}DF(0)$. Then all the eigenvalues of $A$ are less than $1$ in modulus, and hence there exists a constant $\alpha_s$ such that $|\la s^{-1}| < \alpha_s < 1.$ 
\par 
\medskip
\noindent
Define a norm on $\mbb C^k$ by
\[\|z\|_s= \sum_{i=0}^{\infty}\alpha_s^{-i}|A^i (z)|.\]
Note that the operator norm of $A$ with respect to $\|.\|_s$ in $\mbb C^k$ is less than $1$. In fact \[\|s^{-1}Df(0)\|_s=\|A\|_s < \alpha_s.\]
So there exists a sufficiently small $r >0$ such that 
\[\|s^{-1}DF(z)\|_s < \alpha_s<1\]
for every $z \in B_r(0).$ This implies that 
\[\|F(z)-F(z')\|_s<s \|z-z'\|_s\]
for all $z, z' \in B_r(0).$ Let $\De_r$ denote the disc of radius $r$ in $\mbb C.$ Define a map $H_0: \De_r \to B_r(0)$ as
\[H_0(t)=(0,\hdots, 0,t).\]
Let $$H_1(t)=F\circ H_0(t/\la).$$ Since $|\la|>1$ the function $H_1$ is well defined on $\De_r.$ Now from (\ref{3a})
\begin{align*}
H_1(t)&=F\circ H_0(t/\la) \\
&=F(0,\hdots,0,t/\la) \\
&=(\tilde{f}_1(t),\tilde{f}_2(t),\hdots,\tilde{f}_{k-1}(t), t+\tilde{f}_k(t))
\end{align*}
where deg $\tilde{f}_i \ge 2$ for all $1 \le i \le k.$ 
Thus there exists a constant $K>0$ such that 
\begin{align}\label{3b}
\|H_0(t)-H_1(t)\|_s < K |t|^2
\end{align}
for $t \in \De_r.$ Now choose $\rho >0$ such that $\|z\|_s<r/2$ for every $z \in B_\rho(0)$ and 
\begin{align*}
\sum_{n=0}^{\infty}({s}/{|\la|^2})^n K \rho^2< r/2.
\end{align*}
Define $H_n: \De_{\rho} \to \mbb C^k$ inductively as 
\[H_n(t)=F\circ H_{n-1}(t/\la)\]
for every $n \ge 2.$
\par
\medskip
\noindent
{\it Claim:} 
\[\|H_n(t)-H_{n-1}(t)\|_s \le  ( {s}/{|\la|^2})^{n-1} K|t|^2\]
for every $n \ge 1.$
\par
\medskip
\noindent
From (\ref{3b}), the claim is true when $n=1.$ Assume that it holds upto some $n.$ Then
\begin{align*} 
\|H_n(t)\|_s &\le \|H_0(t)\|_s+ \sum_{j=1}^{n-1}\|H_j(t)-H_{j-1}(t)\|_s \\
&< r/2+r/2=r.
\end{align*}
Since $|H_n(t)| \le \|H_n(t)\|_s$, it follows that $H_n(\De_{\rho}) \subset B_r(0).$
Now
\begin{align*}
\|H_{n+1}(t)-H_n(t)\|_s&=\|F\circ H_n(t/\la)-f \circ H_{n-1}(t/\la)\|_s \\
&< s \|H_n(t/\la)-H_{n-1}(t/\la)\|_s \\
&\le  ( {s}/{|\la|^2})^{n-1} K|t|^2.
\end{align*}
This proves the claim. It also shows that the sequence $\{H_n\}$ converges uniformly on $\De_{\rho}$, i.e., $H_n \to H$ such that
\[F\circ H(t)= H(\la t)\]
where $t \in \De_\rho.$
\par
\medskip
\noindent
The function $H$ obtained locally at the origin can be extended to the entire complex plane in the following manner:
\[H(t)=F^{n}\circ H(t/\la^n).\]
Consider $G_n=\pi_k(H_n(t))$ for $n \ge 0.$ Since $G_n'(0)=1$ for all $n \ge 0$, the function $H$ obtained is injective at the origin, and hence on all of $\mbb C.$ Also $G(t)=\pi_k(H(t))$ is locally invertible at the origin. Now as the unstable manifold can be realized locally as a graph of a map and $H(\mbb C) \subset W^u(a)$, it follows that $H(\mbb C)=W^u(a).$
\end{proof}
\par
\medskip
\noindent
It follows that there exists a mapping $H:\mbb C \to \mbb C^k$ such that the map $F$ as in (\ref{3}) with a $(k-1,1)$ saddle point at $a \in \mbb C^k$, the unstable manifold $W^u(a)=H(\mbb C)$ and $H(\la t)=F\circ H(t).$
\begin{proof}[Proof of Theorem \ref{thm1.4}]
Let $\rho=\log d/\log  |\la|$, i.e., $|\la|^{\rho}=d.$ We then compute the type, using this value of $\rho$ and show that the coordinates $h_i$ of $H$ are of mean type. This will allow us to conclude that the $h_i$'s have order $\rho$. By Lemma \ref{3lem(a)} and Theorem \ref{thm1.3},
\begin{align*}
\sigma_i&=\limsup_{r \to \infty}\frac{\log \max_{|t|=r}|h_i(t)|}{r^{\rho}} \\
&\le \limsup_{n \to \infty}\frac{\log \max_{|t|=|\la^{n+1}t_0|}|h_i(t)|}{|\la^n t_0|^{\rho}} \\
&=\limsup_{n \to \infty}\frac{\log \max_{|t|=|\la t_0|}|\pi_i \circ F^n\circ H(t)|}{d^n|t_0|^{\rho}}\\
&\le \limsup_{n \to \infty}\frac{\log \max_{|t|=|\la t_0|}C_i(\ep)^{d^n}\max\{|h_i(t)|^{d^n},M^{d^n}\}}{d^n|t_0|^{\rho}} \\
&=\frac{\log \max_{|t|=|\la t_0|}C_i(\ep)\max\{|h_i(t)|,M\}}{|t_0|^{\rho}} < \infty
\end{align*}
for every $2 \le i\le k.$ 
\par\medskip\no
{\it Claim:} $\sigma_i >0$ for every $2 \le i \le k.$
\par\medskip\noindent
Note that none of the $h_i$'s , $1 \le i \le k$ can be constant. Indeed, if there exists $i_0$  such that $h_{i_0}$ is constant, by Lemma \ref{lem3.1}
we see that each of the $h_i$'s are bounded entire maps on $\mbb C$, and hence they are constant. But this contradicts the fact that $h_k(t)$ is locally injective at the origin.
\par\medskip\noindent
Choose $w_i \in \mbb C$ and $t_0 \in \mbb C$ such that $w_i=h_i(t_0)$ and $|w_i|$ is sufficiently large.
Then from Theorem \ref{thm1.3} it follows that 
\[|w_n^i|\ge C_i(\ep)^{d^n}|w_i|^{d^n}\]
where $w_n^i=\pi_i \circ F^n(w_1,w_2,\hdots,w_k).$ Now
\begin{align*}
\sigma_i&=\limsup_{r \to \infty}\frac{\log \max_{|t|=r}|h_i(t)|}{r^{\rho}} \\
& \ge \limsup_{n \to \infty}\frac{\log |h_i(\la^n t_0)|}{|\la^n t_0|^{\rho}} \\
&=\limsup_{n \to \infty}\frac{\log |\pi_i \circ F^n \circ H(t_0)|}{|\la^n t_0|^{\rho}}\\
&\ge \frac{\log C_i(\ep)+\log|w_i|}{|t_0|^{\rho}} >0.
\end{align*}
\end{proof}
\begin{lem}
Let \[\sigma_1=\limsup_{r \to \infty}\frac{\log \max_{|t|=r}|h_1(t)|}{r^{\rho}} .\]
Then $0 < \sigma_1 < \infty$ and
\begin{align}\label{3e}
\limsup_{r \to \infty}\frac{\log \max_{|t|=r}|\pi_i \circ G_j \circ H(t)|}{r^{\rho}} =d_{[m-k+2]}...d_{[j+i-k]}\sigma_1 
\end{align}
where $1 \le i\le k$ and $0 \le j\le m.$
\end{lem}
\begin{proof}
Let
\[\alpha_i^j=\limsup_{r \to \infty}\frac{\log \max_{|t|=r}|\pi_i \circ G_j \circ H(t)|}{r^{\rho}}\]
for every $1 \le i\le k$ and $0 \le j\le m.$ 
\par\medskip\noindent
Then observe that so far we have proved that
\[0< \alpha_i^0< \infty\]
for every $2 \le i \le k$ and $\alpha_1^0=\sigma_1.$ To show $0 < \sigma_1 < \infty$, note that
\[|\pi_i \circ G_j\circ H(t)|=|\pi_{i+1}\circ G_{j-1}\circ H(t)|\]
for $1 \le i \le k-1$ and $1 \le j\le m$, i.e.,
\begin{align}\label{3d}
\alpha_i^j=\alpha_{i+1}^{j-1}.
\end{align} 
By Lemma \ref{lem3.1} 
\[|\pi_i \circ G_j\circ H(t)|\le (1+\ep)|\pi_{i-1} \circ G_j\circ H(t)|^{d_{[j-k+i]}}+M\]
i.e.,
\[\alpha_i^j \le d_{[j-k+i]} \alpha_{i-1}^j\]
for $2 \le i\le k$ and $0 \le j \le m.$ It follows from (\ref{3d}) that
\begin{align*}
\alpha_1^0 \ge &\frac{\alpha_2^0}{d_{[m-k+2]}}=\frac{\alpha_1^1}{d_{[m-k+2]}} \\
\ge &\frac{\alpha_2^1}{d_{[m-k+2]}d_{[1-k+2]}}=\frac{\alpha_1^2}{d_{[m-k+2]}d_{[1-k+2]}}\\
&\vdots \\
\ge &\frac{\alpha_2^{m-1}}{d_{[m-k+2]}d_{[1-k+2]}\hdots d_{[m-1-k+2]}}=\frac{\alpha_1^m}{d}.
\end{align*}
Also,
\begin{align*}
\alpha_1^m &=\limsup_{r \to \infty}\frac{\log \max_{|t|=r}|\pi_1 \circ F \circ H(t)|}{r^{\rho}} \\
&=\limsup_{r \to \infty}\frac{\log \max_{|t|=r}| H(\la t)|}{|t|^{\rho}} \\
&=\limsup_{r \to \infty}\frac{\log \max_{|t|=r}| H(t)|}{|t|^{\rho}} |\la|^{\rho} \\
&=d\alpha_1^0.
\end{align*}
Thus it follows that
\[\alpha_1^0=\frac{\alpha_1^1}{d_{[m-k+2]}}=\frac{\alpha_1^2}{d_{[m-k+2]}d_{[1-k+2]}}=\hdots = \frac{\alpha_1^m}{d}.\]
Now, as $0 < \alpha_i^0 < \infty$  for $2 \le i \le k $ and $\alpha_2^0=\alpha_1^1$,
\[0 < \alpha_1^0 =\sigma_1 <\infty.\]
Also 
\[\alpha_1^j=\limsup_{r \to \infty}\frac{\log \max_{|t|=r}|\pi_1 \circ G_j \circ H(t)|}{r^{\rho}}=d_{[-k+2]}\hdots d_{[j+1-k]}\sigma_1\]
for $1 \le j \le m.$
\par 
\medskip
\noindent
Assume that (\ref{3e}) is true for some $i$, $1 \le i < k$ and for all $1 \le j \le m$. Then
\begin{align*}
\alpha_{i+1}^j=\alpha_{i}^{j+1}=d_{[-k+2]}\hdots d_{[j+(i+1)-k]}\sigma_1
\end{align*}
for every $1 \le j \le m-1.$ When $j=m,$
\begin{align*}
\alpha_{i+1}^m&=\limsup_{r \to \infty}\frac{\log \max_{|t|=r}|\pi_{i+1} \circ G_m \circ H(t)|}{r^{\rho}} \\
&=d \alpha_{i+1}^0=d \alpha_i^1=d d_{[-k+2]}\hdots d_{[(i+1)-k]} \sigma_1 \\
&=d_{[-k+2]}\hdots d_{[m-1-k+2]}d_{[m-k+2]}\hdots d_{[m+(i+1)-k]} \sigma_1.
\end{align*}
This proves (\ref{3e}), and hence completes the proof of Theorem {\ref{thm1.4}}.
\end{proof}
\medskip
\noindent
We will now prove Theorem {\ref{thm1.5}} following \cite{J}.
Note that $\ti{K}=H^{-1}(K)=H^{-1}(K^+).$ From \cite{BP}, the Green's function
\[G^+(z)=\lim_{n \to \infty} \frac{1}{d^n}\log^+|F^n(z)|\]
is plurisubharmonic and continuous on $\mbb C^k$ and is positive and pluriharmonic on $\mbb C^k \setminus K^+.$
Let $u^+(t)=G^+ \circ H(t).$ Then 
\[\ti{K}=\{t\in \mbb C: u^+(t)=0 \}\]
and
\[u^+(\la t )=d u^+(t)\] where $d=d_1d_2\hdots d_m.$
The order of $u^+$ is
\[ \rho={\rm ord} u^+=\limsup_{r \to \infty}\frac{\log \max_{|t|=r}u^+(t)}{\log r}.\]
\begin{prop}
$\rho={\rm ord}u^+=\log d/ \log |\la|.$
\end{prop}
\begin{proof}
Let $r \gg 1.$ Then there exists $n \ge 0$ such that 
\[|\la|^n < r <|\la|^{n+1}.\]
Since $u^+$ is harmonic on $\mbb C$, the maximum principle shows that
\begin{align*}
\frac{\log \max_{|t|=r}u^+(t)}{\log r} \le \frac{\log \max_{|t|=|\la|^{n+1}}u^+(t)}{\log |\la|^{n}}.
\end{align*}
Hence
\begin{align*}
\rho &\le \limsup_{n \to \infty}\frac{\log \max_{|t|=|\la|^{n+1}}u^+(t)}{\log |\la|^{n}} \\
&=\limsup_{n \to \infty}\frac{\log \max_{|t|=1}d^{n+1}u^+(t)}{n \log |\la|} 
=\frac{\log d}{\log |\la|}.
\end{align*}
Also
\begin{align*}
\rho &\ge \limsup_{n \to \infty}\frac{\log \max_{|t|=|\la|^{n}}u^+(t)}{\log |\la|^{n}} \\
&=\limsup_{n \to \infty}\frac{\log \max_{|t|=1}d^{n}u^+(t)}{n \log |\la|} 
=\frac{\log d}{\log |\la|}=\frac{\log d}{\log |\la|}.
\end{align*}
Hence $\rho=\log d/ \log |\la|.$
\end{proof}
\par
\noindent
We now show that $\mbb C \setminus \tilde{K}$ has finitely many components. For this recall the following
(Theorem 8.9  in \cite{Ha}) fact about subharmonic functions.
\medskip
\par
\no
{\bf Theorem.}
{\it 
Let $u_1(z),u_2(z) \hdots u_N(z)$ be non-constant, non-negative, subharmonic  functions in the plane satisfying
\[u_i u_j(z)\equiv 0\]
for $1 \le i <j \le N$ and $N \ge 2.$ Let
\[B_j(r)=\max_{|z|=r}u_j(z)\]
and 
\[B(r)=\max_{1 \le j \le N} B_j(r).\]
Then \[\liminf_{r \to \infty} \frac{B(r)}{r^{N/2}} > 0.\]
}
\par
\noindent
Since ${\rm ord}u^+=\rho$ there exists an unbounded sequence $\{r_n\}$ of positive real numbers such that
\[\rho=\lim_{n \to \infty} \frac{\log \max_{|t|=r_n} u^+(t)}{\log r_n} .\]
So for an arbitrary $\ep > 0$ 
\begin{align}\label{4a}
\max_{|t|=r_n} u^+(t)\le r_n^{\rho+\ep}
\end{align}
for sufficiently large $n.$ 
\par\medskip\noindent
Suppose $\mbb C\setminus \ti{K}$ has infinitely many components. Choose an integer $N > \max\{2, 2(\rho + \ep)\}$ and let $U_1,U_2,\hdots,U_N$ be a collection of some $N$ components in 
$\mbb C \setminus \ti{K}.$ For $1 \le i \le N$, let
\begin{align*}
u_i(z)=\left\{
           \begin{array}{ll}
             u^+(z),&\hbox{if $z \in U_i$;} \\
             0,&\hbox{otherwise.}
           \end{array}
         \right.
\end{align*}
From (\ref{4a})
\begin{align*}
\frac{\max_{1 \le i \le N} \{\max_{|t|=r_n} u_i(t)\}}{r_n^{\rho + \ep }}=\frac{B(r_n)}{r_n^{\rho+\ep}} \le 1 .
\end{align*}
Thus
\[\lim_{n \to \infty} \frac{B(r_n)}{r_n^{N/2}}=0,\]
which is a contradiction.
\par\medskip\no
This argument also shows that the number of components of $\mbb C \setminus \ti{K}$ can at most be $\max\{1,2\rho\}.$
\begin{thm}
The number of components of $\mbb C\setminus \ti{K}$ cannot exceed $\max\{1, 2\rho\}.$ Therefore every component of $\mbb C \setminus \ti{K}$ is periodic under multiplication by $\la$.
\end{thm}
\begin{proof}
Let the number of components in $\mbb C \setminus \ti{K}$ be $q'$. 
\par \medskip \no
Note that if  $t \in \mbb C\setminus \ti{K}$ then 
$\la^n t \in \mbb C\setminus \ti{K}$ for every $n \ge 1$, as
\begin{align}\label{4c}
u^+(\la^n t)=d^n u^+(t).
\end{align}
Therefore $$\la^n U_i \subset U_{i(n)}$$ for every $n \ge 1$ and $1 \le i \le q'.$ As the number of components in $\mbb C \setminus \ti{K}$ is finite, for every $1 \le i \le q'$ there exists $n_i \ge 1$ such that 
\[\la^{n_i} U_i \subset U_{i}.\] 
Hence each component is periodic.
\end{proof}
\no
Let us recall the notion of an {\it access} from \cite{J}. Let $U$ be an open set and $a$ a point in the boundary of $U.$ We say that $a$ is {\it accessible} from $U$ if there exists a curve $\gamma:[0,1] \to \overline{U}$ which satisfies $\gamma(0)=a$ and $\gamma((0,1]) \subset U.$ We call such a $\gamma$ an {\it access.}
\begin{cor}\label{4cor}
The origin is accessible from an arbitrary component of $\mbb C \setminus \ti{K}$. Moreover the access
$\Gamma$ can be periodic, i.e., if $q$ is the period of the component, there is a path $\Gamma$ satisfying $\Gamma([0,1]) \subset \la^q \Gamma([0,1]).$ 
\end{cor}
\begin{proof}
Pick $t_0 \in \mbb C \setminus \ti{K}$, i.e., $t_0 \in U_{i_0}$ where $1 \le i_0 \le q'$ and let $q_0$ be the period of $U_{i_0}.$ Then $t_0/\la ^{q_0}$ lies in the same component. So there exists a curve $\gamma:[{1}/{2},1]$ such that $\gamma({1}/{2})=t_0/\la^{q_0}$ and $\gamma(1)=t_0.$
Now extend the path to $[0,1]$ by
\begin{align*}
\Gamma(\xi)= \left\{
           \begin{array}{ll}
             0,&\hbox{if $\xi=0$;} \\
             (\la^{q_0})^{-n}\gamma(2^n \xi),&\hbox{if $\frac{1}{2} \le 2^n \xi \le 1$ for $n \ge 0$.}
           \end{array}
         \right.
\end{align*}
Since $\{{t_0}/{(\la^{q_0})^n}\}$ is a sequence converging to the origin, $\gamma$ is a well defined access to 
the origin from $U_{i_0}.$ Also \[\Gamma(\xi)=\la^{q_0} \Gamma(\xi/2).\]
Thus $\Gamma$ is the required curve.
\end{proof}
\begin{lem}
The components of $\mbb C \setminus \ti{K}$ have a natural circular ordering.
\end{lem}
\begin{proof}
Let $U_1,U_2,\hdots, U_{q'}$ be the components of $\mbb C\setminus \ti{K}$ and $\Gamma_1,\Gamma_2, \hdots ,\Gamma_{q'}$ the corresponding accesses. Choose $0<r \le \min\{|\Gamma_1(1)|,|\Gamma_2(1)|, \hdots ,|\Gamma_{q'}(1)|\}.$ Let
\[\xi_j=\min\{0 < \xi \le 1:|\Gamma_j(\xi)|=r\}\]
for each $1 \le j \le q'.$ Then the $\Gamma(\xi_j)$'s are points on the circle $\{z: |z|=r\}$ and hence have a natural circular ordering. Note that the circular ordering is independent of the choice of $\Gamma_j$'s.
\end{proof}
\no
Let $p':\{1,2,\hdots,q'\} \to \mbb N$ be a function such that 
\[\la U_i \subset U_{(i+p'(i)){\rm mod}~ q'}.\]
{\it Claim:} The function $p'$ is  constant.
\par \medskip \no
Note that $\la U_i\cap \la U_j=\emptyset$ for $i \neq j.$ Since there is a circular ordering on the components of $\mbb C \setminus \ti{K}$, we see that
$$1+p'(1) < 2+p'(2) < \hdots < q'+p'(q'), $$
i.e., $p'$ is an increasing function. 
\par \medskip \no 
If $p'$ is non-constant then there exists a $j_0 > 1$ such that
\[j_0=\min\{i: p'(i)>p'(1), ~2 \le i\le q'\}.\]
Then \[1+p'(1)\le (j-1)+p'(j-1) < (j_0-1)+p'(j_0)\]
for every $2 \le j < j_0.$ Since $\la(\mbb C \setminus \ti{K})=\mbb C \setminus \ti{K}$, there exists $i_0> j_0$ such that 
\[i_0+p'(i_0)=(j_0-1)+p'(j_0)\]
i.e., $p'(i_0)<p'(j_0)$, which is a contradiction! 

\par\medskip\no
This shows that the rotation number $p'$ is well defined and hence the period of each component of 
$\mbb C\setminus \ti{K}$ is the same. Let $N,~p,~q$ be as mentioned in Section 1. Then it can be easily seen that $q$ is the period of each component and there are $N$ cycles.
\begin{prop}
The following three conditions are equivalent:
\begin{enumerate}
\item[(i)] $\ti{K}$ is bridged.
\item[(ii)] The component of $\ti{K}$ containing $0$ is unbounded.
\item[(iii)] $\ti{K}$ has an unbounded component.

\end{enumerate}
\end{prop} 
\begin{proof}
If $\ti{K}$ is bridged and $A$ is the component of $\ti{K}$ containing origin, there exists $t \in A$ such that $t \neq 0.$ Since $0 \in A$, $\la^n A=A$ for every $n \ge 0.$ Now from (\ref{4c}) it follows that $\la^n t \in A$ for every $n \ge 0$, and hence $A$ is unbounded. Thus $(i) \Rightarrow (ii) \Rightarrow (iii).$

\par\medskip\no
For the converse suppose that $A$ is bounded and hence $\bar{A}$ is compact. Then there exists a  simple closed curve around the origin $\Gamma$ in $\mbb C \setminus \ti{K}$ such that $A \subset D_{\Gamma}$, where $D_{\Gamma}$ is the domain enclosed by the curve $\Gamma.$ Since $\Gamma$ is closed curve in $\mbb C \setminus \ti{K}$, $\la^j \Gamma$ is a curve in $\mbb C \setminus \ti{K}$, for every $j \ge 1.$ Now, if we assume $(iii)$,then there exists an unbounded component in $\ti{K}$, say $A_{\infty}.$ But $\la^j \Gamma \cap A_{\infty} \neq \emptyset $ for sufficiently large $j \ge 0$, which is a contradiction! Hence $(iii) \Rightarrow (ii) \Rightarrow (i).$ 
\end{proof}
\begin{figure}[H]\label{fig5}
\includegraphics[height=3in,width=6.5in]{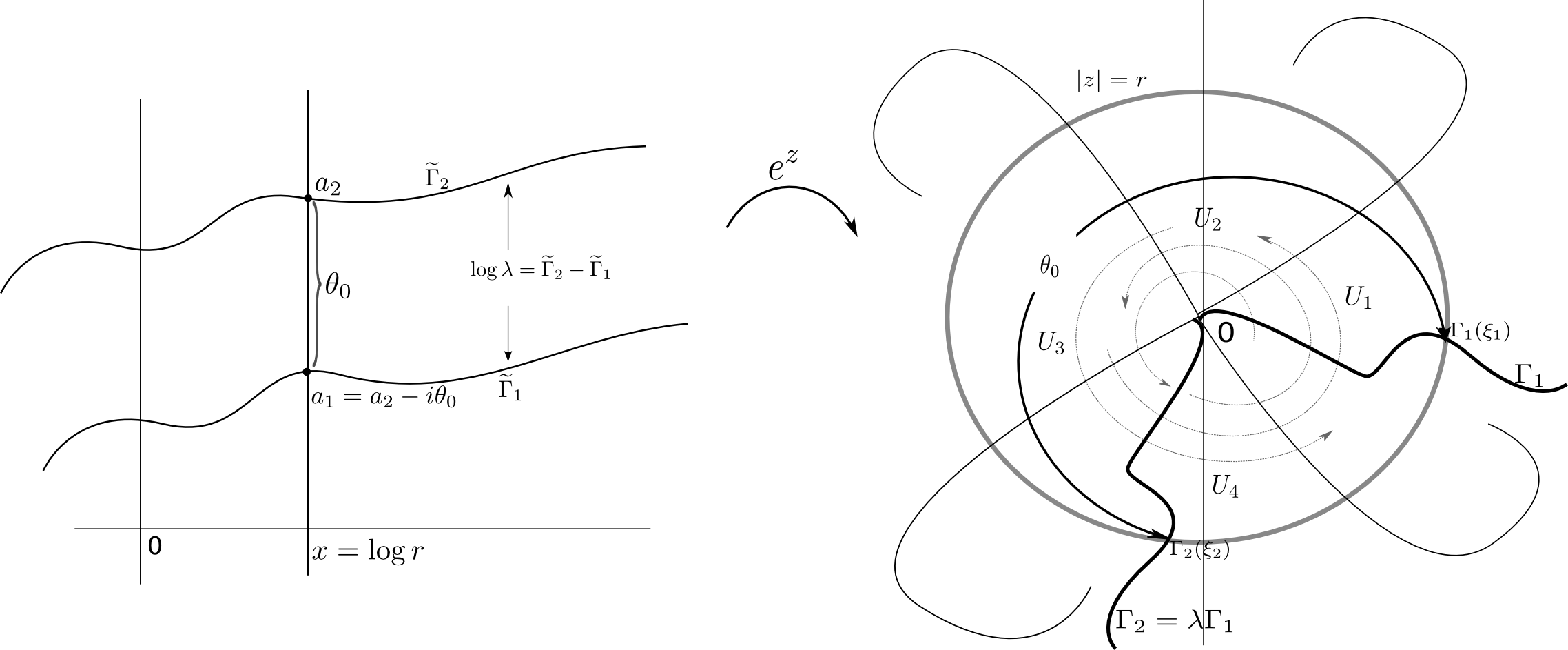}
\caption{}
\end{figure}
\begin{proof}[Proof of Theorem 1.5]
The connected components of $\mbb C \setminus \ti{K}$ can be decomposed into $N$ cycles. Choose one component from each of these cycles and let the resulting collection be relabelled as
${U}_1,{U}_2,\hdots ,{U}_N.$
\par\medskip\no
Define 
\[v(t)=\max\{u^+(t)-1,0\}.\]
Let $$D=\{t \in \mbb C: v(t)>0\}$$
and then define $D_j=D\cap U_j$ for every $1 \le j\le N.$
\par\medskip\no
Since $\ti{K}$ is bridged with an unbounded component containing the origin, all branches of the logarithm are well defined on $\mbb C \setminus \ti{K}$

\par
\medskip
\noindent
Let $\gamma_1:[0,1]$ be a periodic access  to origin in $U_1$ as constructed in Corollary \ref{4cor}. Extend it to $[0,\infty)$ in the following way
\begin{align*}
\Gamma_1=\left\{
           \begin{array}{ll}
             \gamma_1,&\hbox{if $0 \le \xi \le 1$;} \\
             (\la^q)^n \gamma_1({\xi}/{2^n}),&\hbox{for $n \ge 1 $ and $2^{n-1} \le \xi \le 2^n$.}
           \end{array}
         \right.
\end{align*}
Let $\Gamma_2 = \la \Gamma_1 \subset U_{1+p'}$. Choose $r>0$, and let 
\[\xi_1=\min\{\xi: |\Gamma_1(\xi)|=r\}\]
\[\xi_2=\min\{\xi: |\Gamma_2(\xi)|=r\}.\]
Let $$\theta_0=\arg(\Gamma_2(\xi_2)/\Gamma_1(\xi_1))$$ where $0 \ge \theta_0 <2\pi.$
Fix $a_2 \in \mbb C$ such that $e^{a_2}=\Gamma_2(\xi_2).$ Lift $\Gamma_2$ to a path $\ti{\Gamma}_2:(0,\infty) \to \mbb C$ such that $\ti{\Gamma}_2(\xi_2)=a_2$ using the covering of $\mbb C\setminus \{0\}$
by the exponential map.
This correspond to choosing a branch of the logarithm such that $a_2=\log \Gamma_2(\xi_2).$
Let \[a_1=a_2-i \theta_0.\]
Since $\theta_0 \in [0,2\pi]$, we may use the same branch of the logarithm to lift $\Gamma_1$ to a path
$\ti{\Gamma}_1:(0,\infty)\to \mbb C$ such that $\ti{\Gamma}_1(\xi_1)=a_1.$ Now
\[e^{\ti{\Gamma}_2(\xi)}=\Gamma_2=\la \Gamma_1=\la e^{\ti{\Gamma}_1(\xi)}. \]
This shows that 
\[e^{\ti{\Gamma}_2(\xi)-\ti{\Gamma}_1(\xi)} = \la\] and we may now define
\[\tau =\log \la = \ti{\Gamma}_2(\xi)-\ti{\Gamma}_1(\xi).\]
This incidentally shows that the lift $\ti{\Gamma}_2$ is a translate of $\ti{\Gamma}_1.$
We will henceforth work with this choice of the logarithm.
\par\medskip\no
Define $D_j'=\log D_j$ for $1 \le j \le N$
and let $f(s)=s+q \tau-2\pi ip.$ 
\par\medskip\no 
{\it Claim:} If $s \in D_j'$ then $f(s) \in D_j'.$
\par\medskip\no
Suppose the ordering in the components is anticlockwise, i.e., $0<\theta_0 <2\pi.$ Let $s \in D_j'$. Then $e^s \in D_j$ and therefore $$u^+(\la^q e^s)=d^q u^+(e^s)>1.$$ Hence 
$$\la^q e^s=e^{\tau q+s} \in D_j.$$ But observe that $p'q=q'p$ which means that a  point $s \in D_j$ completes $p$-cycles around the origin to get back to $D_j$, i.e., $\tau q+s \in D_j'+2\pi ip.$ Thus the claim.
\par\medskip\noindent
Note that $\Re \tau=\log |\la| >0$ since $|\la|>1.$ This shows that
\begin{align}{\label{4d}}
\Re f(s)>\Re s. 
\end{align}
\par\medskip\noindent
{\it Claim:} For every $1 \le j\le N$, the domains $D_j'$'s are distributed along the direction $\tau-2 \pi i p/q$, i.e., for every point $z \in D_j'$ the distance of $z$ from the line through the origin in the direction $\tau-2 \pi i p/q$ is bounded. 
\par\medskip\noindent
Suppose $s_j \in D_j'$. Let $D_j'=D^1_j \cup D^2_j$ where 
\[D^1_j=\Big\{z \in D_j': \Re z \le \Re s_j\Big \}\]
and 
\[D^2_j=D_j'\setminus D_j^1.\]
Since the origin is not a limit point of the $D_j$'s it follows that $$m_j=\min\{\Re z: z \in \overline{D_j'} \} > -\infty.$$ Let $a_j \in \overline{D_j'}$ be such that $\Re a_j=m_j.$ Then there exists a curve in $\overline{D_j'}$ joining $a_j$ and $s_j$. Call it $\beta_j^1.$ Define
\[[{\beta}_j^1]=\Big\{z \in \mbb C: z \in (\beta_j^1(t)-2\pi i,\beta_j^1(t)+2\pi i), 0\le t\le 1\Big\}.\]
Observe that a circle around the origin is mapped to segment of length $2 \pi$ parallel to the $y$-axis by the branch of the $\log$ function defined above. 
So $D_j^1 \subset [{\beta}_j^1]$ and hence it is bounded, with a finite distance from the line $(\tau -2i \pi p/q)t, ~t \in(0,\infty).$ 

\par\medskip\no
If $s_j \in D_j^2 \subset D_j'$ then from (\ref{4d})$$f(s_j)=s_j+q (\tau-i2\pi p/q) \in D_j^2 \subset D_j'.$$ There exists a curve say $\beta_j^2 : [0,1] \rightarrow D'_j$ 
connecting $s_j$ and $f(s_j).$ The curve $\beta_j^2$ can be extended to $[0,\infty)$ in the following way:
\begin{align*}
\beta_j^2(\xi)=\left\{
           \begin{array}{ll}
             \beta_j^2,&\hbox{if $0 \le \xi \le 1$;} \\
             f^n \circ \beta_j^2(\xi-n),&\hbox{if $n\le \xi \le n+1$ for some $n \ge 1 $.}
           \end{array}
         \right.
\end{align*} 
Observe that for any point $s \in \mbb C$ the distance of $s$ and $f(s)$ from the line $(\tau -2i \pi p/q)t, ~t \in(0,\infty)$ is the same. So
$\beta_j^2$ is distributed along the direction of $\tau -2i \pi p/q.$ Now define
\[[{\beta}_j^2]=\Big\{z \in \mbb C: z \in (\beta_j^2(t)-2\pi i,\beta_j^2(t)+2\pi i), t \in [0,\infty)\Big\}.\] 
Then $D_j^2 \subset [{\beta}_j^2].$ Hence the claim.
\par\medskip\no 
{\it Claim:} There exists $C>0$ such that
\begin{align}\label{4e}
\bigg \vert{\Re s-\frac{\Re \tau}{|\tau-2i \pi p/q |}|s|}\bigg \vert < C
\end{align}
for $s \in \bigcup D_j'.$
\par\medskip\no
Let $$\alpha=\cos^{-1}\bigg(\frac{\Re \tau}{|\tau-2i \pi p/q|}\bigg).$$
Note that if $s \in \bigcup D_j'$ then $|s||\sin(\theta-\alpha)|$ is the distance of the point from the line $(\tau -2i \pi p/q)t, ~t \in(0,\infty)$, and by the previous claim there exist $M>0$ such that
\[|s||\sin(\theta-\alpha)|< M\]
where $$\theta=\cos^{-1}\bigg(\frac{\Re s}{|s|}\bigg).$$
Since the $D_j'$s are unbounded domains, $\sin(\theta-\alpha) \to 0$ as $|s| \to \infty$, i.e., if $|s|>2M$ then 
\[|\sin(\theta-\alpha)|<1/2\]
which means 
$$\theta-\alpha \in [-{\pi}/{6} ,{\pi}/{6} ] \subset( -{\pi}/{2},{\pi}/{2}).$$
On $( -{\pi}/{2},{\pi}/{2})$ 
\[|{2t}/{\pi}|\le|\sin t| \le |t| .\]
This shows that when $|s|>2M,$
\begin{align*}
|s|\bigg \vert{\frac{\Re s}{|s|}-\frac{\Re \tau}{|\tau-2i \pi p/q |}}\bigg \vert  
=~&2|s| \bigg|\sin\bigg(\frac{\theta-\alpha}{2}\bigg)\bigg|\bigg|\sin\bigg(\frac{\theta+\alpha}{2}\bigg)\bigg|\\
<&~2|s|\bigg|\frac{\theta-\alpha}{2}\bigg|\bigg|\sin\bigg(\frac{\theta+\alpha}{2}\bigg)\bigg| \\
<&~2|\pi||s||\sin(\theta-\alpha)|\bigg|\sin\bigg(\frac{\theta+\alpha}{2}\bigg)\bigg|<+\infty.
\end{align*}
When $|s|\le 2M$, (\ref{4e}) follows just from the triangle inequality.
\par\medskip\no
For a path $\sigma$ in $\mbb C$ let ${\rm len}(\sigma)$ denote the length of $\sigma.$ Since $\la^n U_j=U_{(j+np'){\rm mod}~ q'}$, it follows from the above calculation that
\begin{align}{\label{4f}}
{\rm len}\bigg( \{\Re s=\xi\}\cap \bigcup_{j=1}^N \bigcup_{n=0}^{q-1} \log(\la^n U_j)\bigg) \le 2 \pi.
\end{align}
Now note that 
\begin{align*}
{\rm len}\bigg( \{\Re s=\xi\}\cap \bigcup_{j=1}^N  \log( U_j)\bigg)=&{\rm len}\bigg( \{\Re s=\xi+\Re \tau\}\cap \bigcup_{j=1}^N  \log( \la U_j)\bigg) \\
=&{\rm len}\bigg( \{\Re s=\xi+2\Re \tau\}\cap \bigcup_{j=1}^N  \log( \la^2 U_j)\bigg) \\
& \vdots\\
=&{\rm len}\bigg( \{\Re s=\xi+(q-1)\Re \tau\}\cap \bigcup_{j=1}^N  \log( \la^{q-1} U_j)\bigg).
\end{align*}
Combining this with (\ref{4f}), we get
\begin{align*}
2\pi q \Re \tau  \ge &\int\limits_{\xi}^{\xi+q \Re \tau}{\rm len}\bigg( \{\Re s=\xi\}\cap \bigcup_{j=1}^N \bigcup_{n=0}^{q-1} \log(\la^n U_j)\bigg) d\xi \\
=&\int\limits_{\xi}^{\xi+q \Re \tau}{\rm len}\bigg( \{\Re s=\xi\}\cap \bigcup_{j=1}^N  \log(U_j)\bigg) d\xi +\\&\int\limits_{\xi-\Re \tau}^{\xi+(q -1)\Re \tau}{\rm len}\bigg( \{\Re 
s=\xi\}\cap \bigcup_{j=1}^N \log(U_j)\bigg) d\xi +\\
&\vdots \\ +&\int\limits_{\xi-(q-1)\Re \tau}^{\xi+\Re \tau}{\rm len}\bigg( \{\Re s=\xi\}\cap \bigcup_{j=1}^N  \log(U_j)\bigg) d\xi ,
\end{align*}
i.e.,
\begin{align}\label{4g}
2\pi q \Re \tau  \ge \sum_{n=0}^{q-1}\int\limits_{\xi-n \Re \tau}^{\xi+(q-n) \Re \tau}{\rm len}\bigg( \{\Re s=\xi\}\cap \bigcup_{j=1}^N  \log( U_j)\bigg) d\xi .
\end{align}
Since the $\log (U_j)$'s are invariant under the map $s \mapsto f(s)=s+q\tau-2 i\pi p $ we have
\begin{align*}
{\rm len}\bigg( \{\Re s=\xi\}\cap \bigcup_{j=1}^N  \log( U_j)\bigg)={\rm len}\bigg( \{\Re s=\xi+q \Re \tau\}\cap \bigcup_{j=1}^N  \log( U_j)\bigg).
\end{align*}
So for any $n$, $1 \le n \le {q-1} $ 
\begin{align}\label{4h}
\int\limits_{\xi-n \Re \tau}^{\xi}{\rm len}\bigg( \{\Re s=\xi\}\cap \bigcup_{j=1}^N  \log(\la^n U_j)\bigg) d\xi =\int\limits_{\xi+(q-n) \Re \tau}^{\xi+q \Re \tau}{\rm len}\bigg( \{\Re 
s=\xi\}\cap \bigcup_{j=1}^N  \log( U_j)\bigg) d\xi.
\end{align}
From (\ref{4g}) and (\ref{4h}) we get that
\begin{align*}
2\pi q \Re \tau  &\ge q\int\limits_{\xi}^{\xi+q \Re \tau}{\rm len}\bigg( \{\Re s=\xi\}\cap \bigcup_{j=1}^N  \log( U_j)\bigg) d\xi  \\
&\ge  q\int\limits_{\xi}^{\xi+q \Re \tau}{\rm len}\bigg( \{\Re s=\xi\}\cap \bigcup_{j=1}^N  D_j'\bigg) d\xi.
\end{align*}
Therefore
\[\frac{1}{q \Re \tau}\int\limits_{\xi}^{\xi+q \Re \tau}{\rm len}\bigg( \{\Re s=\xi\}\cap \bigcup_{j=1}^N  D_j'\bigg) d\xi \le \frac{2\pi}{q}.\]
\medskip\no 
Define subharmonic functions $\{w_j\}_{j=1}^N$ as follows:
\begin{align*}
w_j(s)=\left\{
           \begin{array}{ll}
             v(e^s),&\hbox{if $s \in D_j'$;} \\
             0,&\hbox{otherwise.}
           \end{array}
         \right.
\end{align*}
For each $1 \le j \le N$, define a function $\theta_j(r):(0, \infty) \to [0, \infty]$ as follows. Restrict $w_j$ to $S_r$, the circle of radius $r$ around the origin and let $N_j(r) 
\subset S_r$ be the collection of of those points $re^{i \psi}$ such that $w_j$ has a positive lower bound on the arc $(re^{i (\psi+\ep)},re^{i (\psi-\ep)})$ for some $\ep>0.$ 
Let $s \in \overline{N_j(r)}$. Then
\begin{align*}
\theta_j(r)=\left\{
           \begin{array}{ll}
             0,&\hbox{if $w_j(s)> 0$ for some $s \in S_r  \setminus N_j(r) $;} \\
             \infty, &\hbox{if $S_r \subset N_j(r)$ or $w_j|_{S_r}\equiv 0$;}
           \end{array}
         \right.
\end{align*}
else
\begin{align*}
\theta_j(r)= \frac{1}{r}\left(\mbox{maximum length of the components of } N_j(r) \cap S_r \right)
\end{align*}
otherwise.
Choose $r_0$ such that if $r >r_0$, $\theta_j(r) < \infty $ for every $1 \le j \le N.$
Let $r$ be such that $r_0< \ka r$ where $1/e \le \ka < 1.$ Then it follows from Tsuji's inequality that
\begin{align*}
\log \max_{|t|=r} w_j(t) \ge \pi \int\limits_{r_0}^{\ka r}\frac{dr}{r\theta_j(r)} +\log \max_{|t|=\ka r_0}w_j(t)+ 1/6 \;\log (1-\ka)^{3/2} 
\end{align*}
and hence
\begin{align}\label{4k}
\sum_{j=1}^n \log \max_{|t|=r} w_j(t) \ge \sum_{j=1}^n \pi \int\limits_{r_0}^{\ka r} \frac{dr}{r\theta_j(r)} -{\rm const}.
\end{align}
\no 
The Cauchy--Schwarz inequality shows that
\[N^2=\bigg(\sum_{j=1}^N \frac{\sqrt{r\theta_j(r)}}{\sqrt{r\theta_j(r)}}\bigg)^2 \le
\bigg(\sum_{j=1}^N r \theta_j(r)\bigg) \bigg(\sum_{j=1}^N \frac{1}{r \theta_j(r)}\bigg) \]
and
\[(\ka r-r_0)^2=\bigg(\int\limits_{r_0}^{\ka r} \frac{\sqrt{r\theta_j(r)}}{\sqrt{r\theta_j}} dr\bigg)^2 \le
\bigg(\int\limits_{r_0}^{\ka r} r\theta_j(r)\bigg) \bigg(\int\limits_{r_0}^{\ka r} \frac{1}{r \theta_j(r)}\bigg) .\]
Therefore
\begin{align}\label{4i}
\sum_{j=1}^N \int\limits_{r_0}^{\ka r} \pi \frac{dr}{r \theta_j(r)} \ge \pi N^2 \int\limits_{r_0}^{\ka r} \frac{dr}{\sum_{j=1}^N r \theta_j(r)} \ge \frac{\pi N^2 (\ka 
r-r_0)^2}{\int\limits_{r_0}^{\ka r} \sum_{j=1}^N r\theta_j(r)dr}.
\end{align}
Since the domains $D'_j$ are directed along $\tau -2 \pi i p/q$, for sufficiently large $r>0$, the area of $\bigcup_{j=1}^N D_j' $ upto radius $r$ is bounded above by the area of 
$\bigcup_{j=1}^N D_j'$ upto the real coordinate $({\Re \tau}/{|\tau -2\pi i p/q|})r$. See the Figure 5 for details.
\begin{figure}[H]\label{fig6}
\includegraphics[height=3in,width=3in]{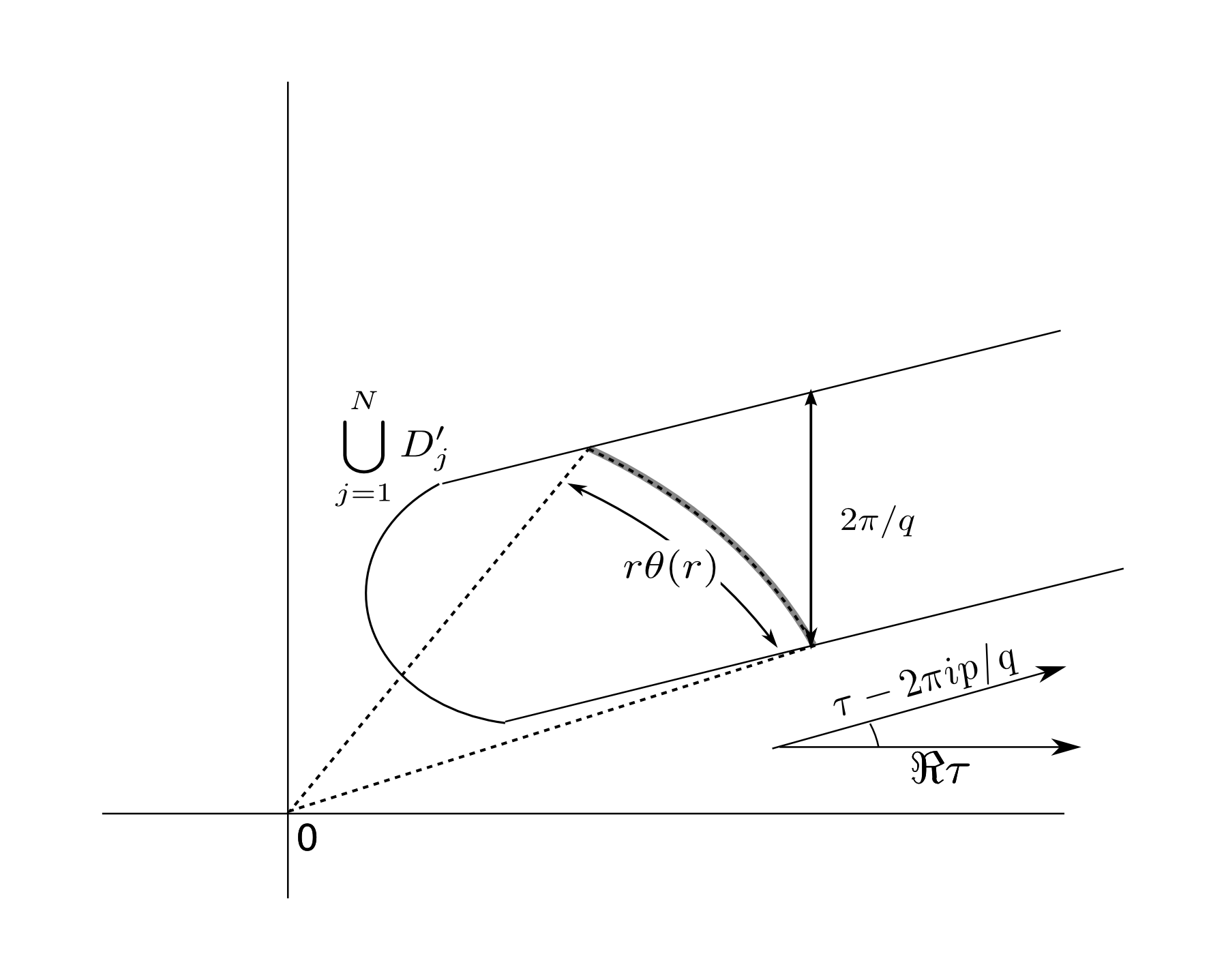}
\caption{}
\end{figure}
\par\medskip\no 
Therefore from (\ref{4g}) it follows that 
\begin{align*}
\int\limits_{r_0}^{\ka r} \sum_{j=1}^N r \theta_j(r)dr &\le \int\limits_0^{\frac{\Re \tau}{|\tau-2\pi i p/q|}\ka r} {\rm len}\bigg(\{\Re s =\xi \cap \bigcup_{j=1}^N D_j'\} \bigg)d\xi + 
{\rm const}  \\
& \le \frac{2 \pi }{q}\frac{\Re \tau}{|\tau-2\pi i p/q|}\ka(r+{\rm const})
\end{align*}
By substituting the above inequality in (\ref{4i}) we get
\begin{align}\label{4j}
\sum_{j=1}^N \pi \int\limits_{r_0}^{\ka r}  \frac{dr}{r \theta_j(r)} \ge \frac{N^2q| \tau-2 \pi i p/q|}{2 \Re \tau} \frac{(\ka r-r_0)^2}{\ka (r +{\rm const.})}.
\end{align}
Since $|e^s| = e^{\Re s}$, the estimate on the left side of (\ref{4k}) can be modified as
\begin{align}\label{4l}
\sum_{j=1}^N \log \max_{|s|=r}w_j(s) &\le N \log \max_{s \in \bigcup D_j', |s|=r}  v(e^s) \nonumber \\
& \le \max_{s \in \bigcup D_j', |s|=r} N \Re s 
\bigg(\frac{\log |u^+(e^s)|}{\log |e^s|}\bigg).
\end{align}
So from (\ref{4k}), (\ref{4j}) and (\ref{4l}) we get 
\begin{align*}
\max_{s \in \bigcup D_j', |s|=r} N \Re s 
\bigg(\frac{\log |u^+(e^s)|}{\log |e^s|}\bigg) \ge \frac{N^2q| \tau-2 \pi i p/q|}{2 \Re \tau} \frac{(\ka r-r_0)^2}{\ka (r +{\rm const.})} -{\rm const.}
\end{align*}
and by dividing both sides by $r$, we get
\begin{align}\label{4m}
\max_{s \in \bigcup D_j', |s|=r} N \frac{\Re s}{|s|}
\frac{\log |u^+(e^s)|}{\log |e^s|} \ge \frac{N^2q| \tau-2 \pi i p/q|}{2 \Re \tau} \frac{(\ka r-r_0)^2}{\ka r(r +{\rm const.})} -\frac{{\rm const.}}{r}.
\end{align}
Earlier we have proved that 
\begin{align*}
\bigg \vert{\Re s-\frac{\Re \tau}{|\tau-2i \pi p/q |}|s|}\bigg \vert
\end{align*}
is bounded, i.e.,
\begin{align*}
\lim_{|s|\to \infty}\frac{\Re s}{|s|}=\frac{\Re \tau}{|\tau-2i \pi p/q |}.
\end{align*}
By taking the limit as $r \to \infty$ in (\ref{4m}) we get
\begin{align}
N\frac{\Re \tau}{|\tau-2i \pi p/q |} {\rm ord}u^+ \ge \frac{\ka N^2q| \tau-2 \pi i p/q|}{2 \Re \tau}.
\end{align}
Since ${\rm ord}u^+=\log d / \Re \tau$ and $1/e \le \ka < 1$ is arbitrary we get
\begin{align*}
\frac{N \log d}{|\tau-2i \pi p/q |} \ge \frac{ N^2q| \tau-2 \pi i p/q|}{2 \Re \tau}
\end{align*}
which reduces to the Yoccoz inequality.
\end{proof}

\section{Appendix}
\no
For a given $\delta > 0$ and real numbers $a_1,a_2,\hdots,a_k$ define a family of subsets $X(i_2,i_3,\hdots ,i_{k-2}) \subset \mbb{C}^k $ where $i_j \in\{0,1\}$ for every $j \in \{2,\hdots,k-2\}$ as :
\[
 X(i_2,i_3,\hdots,i_{k-2}) =\Big\{z\in \mbb{C}^k:\Re\pi_{j-1}(z) \leqslant a_{j-1}-\delta \;\mbox{if}\;i_j=0 \;\mbox{else}\; \Re\pi_{j-1}(z) \geqslant a_{j-1}+\delta \;\mbox{if} \;i_j=1 \Big\}.
\]
\no
Note that $X(i_2,i_3,\hdots,i_{k-2})$ is non-empty provided $k>3$.

\medskip
\begin{thm}\label{thm1}
Let $\rho(z)=\min\Big\{|\pi_1(z)|,\hdots,|\pi_p(z)|\Big\}$. Let $\lambda_1,\lambda_2,\hdots,\lambda_k,a_1,a_2,\hdots,a_k$ be real constants and let $\delta>0,~\eta>0.$
Define a family of maps $\alpha_j:\{0,1\} \to \{0,\lambda_j\}$ as $\alpha_j(0)=0~,\alpha_j(1)=\lambda_j$ for all $j=1,2,\hdots,k.$ Then there
exists a map $\theta \in Aut_1(\mbb{C}^k)$ such that
\begin{enumerate}
 \item [(i)]$|\theta(z)-z+(\alpha_1(1),\alpha_2(i_2),\hdots,\alpha_{k-2}(i_{k-2}),\alpha_{k-1}(1),\alpha_k(1))|<\eta e^{-\rho(z)^{1/4}}$ whenever
 \[
 z \in X(i_2,i_3,\hdots,i_{k-2}) \cap \Big\{z \in \mbb{C}^k: \Re \pi_{k-2}(z)\geqslant a_{k-2}+\delta,~\Re \pi_{k-1}(z)\geqslant a_{k-1}+\delta,~\Re \pi_{k}(z)\geqslant a_{k}+\delta\Big\}.
 \]
 \medskip
 \item [(ii)]$|\theta(z)-z+(\alpha_1(1),\alpha_2(i_2),\hdots,\alpha_{k-2}(i_{k-2}),\alpha_{k-1}(0),\alpha_k(1))|<\eta e^{-\rho(z)^{1/4}}$ whenever
     \[
 z \in X(i_2,i_3,\hdots,i_{k-2}) \cap \Big\{z \in \mbb{C}^k: \Re \pi_{k-2}(z)\leqslant a_{k-2}-\delta,~\Re \pi_{k-1}(z)\geqslant a_{k-1}+\delta,~\Re \pi_{k}(z)\geqslant a_{k}+\delta\Big\}.
 \]
 \medskip
 \item [(iii)]$|\theta(z)-z+(\alpha_1(0),\alpha_2(i_2),\hdots,\alpha_{k-2}(i_{k-2}),\alpha_{k-1}(1),\alpha_k(1))|<\eta e^{-\rho(z)^{1/4}}$ whenever
 \[z \in X(i_2,i_3,\hdots,i_{k-2}) \cap \Big\{z \in \mbb{C}^k: \Re \pi_{k-2}(z)\geqslant a_{k-2}+\delta,~\Re \pi_{k-1}(z)\geqslant a_{k-1}+\delta,~\Re \pi_{k}(z)\leqslant a_{k}-\delta\Big\}.\]
 \medskip
 \item [(iv)]$|\theta(z)-z+(\alpha_1(0),\alpha_2(i_2),\hdots,\alpha_{k-2}(i_{k-2}),\alpha_{k-1}(0),\alpha_p(1))|<\eta e^{-\rho(z)^{1/4}}$ whenever
 \[z \in X(i_2,i_3,\hdots,i_{k-2}) \cap \Big\{z \in \mbb{C}^k: \Re\pi_{k-2}(z)\leqslant a_{k-2}-\delta,~\Re \pi_{k-1}(z)\geqslant a_{k-1}+\delta,~\Re\pi_{k}(z)\leqslant a_{k}-\delta\Big\}.\]
 \medskip
 \item [(v)]$|\theta(z)-z+(\alpha_1(1),\alpha_2(i_2),\hdots,\alpha_{k-2}(i_{k-2}),\alpha_{k-1}(1),\alpha_k(0))|<\eta e^{-\rho(z)^{1/4}}$ whenever
\[z \in X(i_2,i_3,\hdots,i_{k-2}) \cap \Big\{z \in \mbb{C}^k: \Re\pi_{k-2}(z)\geqslant a_{k-2}+\delta,~\Re \pi_{k-1}(z)\leqslant a_{k-1}-\delta,~\Re\pi_{k}(z)\geqslant a_{k}-\lambda_k+\delta\Big\}.\]
\medskip
 \item [(vi)]$|\theta(z)-z+(\alpha_1(1),\alpha_2(i_2),\hdots,\alpha_{k-2}(i_{k-2}),\alpha_{k-1}(0),\alpha_k(0))|<\eta e^{-\rho(z)^{1/4}}$ whenever
 \[z \in X(i_2,i_3,\hdots,i_{k-2}) \cap \Big\{z \in \mbb{C}^k: \Re\pi_{k-2}(z)\leqslant a_{k-2}-\delta,~\Re \pi_{k-1}(z)\leqslant a_{k-1}-\delta,~\Re\pi_{k}(z)\geqslant a_{k}-\lambda_k+\delta\Big\}.\]
 \medskip
 \item [(vii)]$|\theta(z)-z+(\alpha_1(0),\alpha_2(i_2),\hdots,\alpha_{k-2}(i_{k-2}),\alpha_{k-1}(1),\alpha_k(0))|<\eta e^{-\rho(z)^{1/4}}$ whenever
 \[z \in X(i_2,i_3,\hdots,i_{k-2}) \cap \Big\{z \in \mbb{C}^k: \Re\pi_{k-2}(z)\geqslant a_{k-2}+\delta,~\Re \pi_{k-1}(z)\leqslant a_{k-1}-\delta,~\Re\pi_{k}(z)\leqslant a_{k}-\lambda_k-\delta\Big\}.\]
 \medskip
 \item [(viii)]$|\theta(z)-z+(\alpha_1(0),\alpha_2(i_2),\hdots,\alpha_{k-2}(i_{k-2}),\alpha_{k-1}(0),\alpha_k(0))|<\eta e^{-\rho(z)^{1/4}}$ whenever
 \[z \in X(i_2,i_3,\hdots,i_{k-2}) \cap \Big\{z \in \mbb{C}^k: \Re\pi_{k-2}(z)\leqslant a_{k-2}-\delta,~\Re \pi_{k-1}(z)\leqslant a_{k-1}-\delta,~\Re\pi_{k}(z)\leqslant a_{k}-\lambda_k-\delta\Big\}.\]
\end{enumerate}
\end{thm}
\begin{proof}
As in Section 2, we use the Arakelian--Gauthier theorem to establish the existence of entire functions with the following properties:
\begin{align*}
 \mbox{(i)} \hspace{10 mm} &|f_n^1(z)-\lambda_1| \leqslant\frac{1}{n}
e^{-|z|^{1/4}} <\eta e^{-|z|^{1/4}} ~~\mbox{if}~~\Re z\geqslant
a_k-\lambda_k+\delta/2 ,~\mbox{and} \\
&|f_n^1(z)-0| \leqslant \frac{1}{n} e^{-|z|^{1/4}}<\eta
e^{-|z|^{1/4}}~~\mbox{if}~~ \Re z\leqslant a_k-\lambda_k-\delta/2. &\\
\mbox{(ii)} \hspace{10 mm} &|f_n^j(z)-\lambda_n| \leqslant \frac{1}{n}
e^{-|z|^{1/4}} <\eta e^{-|z|^{1/4}}~~\mbox{if}~~\Re z\geqslant a_{j-1}+\delta/2
, ~\mbox{and}\\
&|f_n^j(z)-0| \leqslant \frac{1}{n} e^{-|z|^{1/4}} <\eta
e^{-|z|^{1/4}}~~\mbox{if}~~\Re z\leqslant a_{j-1}-\delta/2, ~~\mbox{where}~~
2 \le j \le k.
\end{align*}
\no
Define $\theta_n \in {\rm Aut}_1(\mbb C^k)$ as $\theta_n(z_1,z_2,\hdots,z_k)=(v_1,v_2,\hdots,v_k)$ as
\begin{align*}
v_1 &= z_1 + f_{n}^1(v_k),\\
v_2 &= z_2 + f_{n}^2(z_1), \\
&\vdots\\
v_k &= z_k + f_{n}^k(z_{k-1}).
\end{align*}
Pick $z \in \mbb{C}^k$ such that $z \in X(i_2,i_3,\hdots,i_{k-2})$ for $i_j \in \{0,1\}$ for every $1 \le j\le k-2$ and $\Re\pi_{k-2}(z)\ge a_j+\delta ,~\Re\pi_{k-1}(z)\ge a_j+\delta ,~\Re\pi_k(z)\geqslant
a_k-\lambda_k+\delta.$
Then $$|f_n^j(z_{j-1})-\alpha_j(i_{j})|=|\pi_j(\theta_n(z))-\pi_j(z)+\alpha_j(i_j)|< \eta
e^{-\rho(z)^{1/4}}$$
for $2 \le j \le k-2$ and
$$|f_n^j(z_{j-1})-\lambda_j|=|\pi_j(\theta_n(z))-\pi_j(z)+\lambda_j|< \eta
e^{-\rho(z)^{1/4}}$$ for $j=k-1,k.$ Therefore
$$\Re v_k\geqslant \Re z_k-\lambda_k-\eta\geqslant a_k-\lambda_k+\delta/2$$
and hence
$$|\pi_1(\theta_n(z))-\pi_1(z)+\lambda_1|=|f_n^1(v_k)-\lambda_1| <
e^{-|v_k|^{1/4}}/n.$$ 
Since $|v_k|>|z_k|-|\lambda_k|-1$ and $t^{1/4}-(t-\lambda_3
-1)^{1/4}\to 0$ as $t \to \infty$ there exists a constant $C>0$ such that
\[e^{-|v_k|^{1/4}}<C e^{-|z_k|^{1/4}}.\]
If we choose $n$ large enough so that
$C/n<\eta$ and $1/n<\eta$, then $\theta_n$ satisfies condition (i) of the theorem.
Similarly it can be shown that for large $n$, $\theta_n$ satisfies all the other
conditions.
\end{proof}

\no In analogy with Theorem \ref{thm1a} which has $2^3 = 8$ regions where the map $\theta$ has apriori prescribed behaviour, one would imagine that there should be $2^k$ regions in 
Theorem \ref{thm1} with specified behaviour for the map obtained there. However, the behaviour of that map can be summarised within the conditions (i) - (viii) because of its shift-like 
nature. The set $ X(i_2,i_3,\hdots,i_{k-2})$ is the region where the shift--like behaviour is similar for the coordinates $z_j$ with $1 \le j \le k-3$.

\end{document}